\documentclass{amsart}

\usepackage{amsmath}
\usepackage{amsfonts}
\usepackage{amssymb}
\usepackage{mathtools}
\usepackage{dsfont}
\usepackage{algorithmic}
\usepackage{algorithm}
\usepackage{array}
\usepackage[caption=false,font=normalsize,labelfont=sf,textfont=sf]{subfig}
\usepackage{url}
\usepackage{graphicx}
\usepackage{xcolor}

\DeclareMathOperator*{\argmax}{argmax}

\theoremstyle{remark}
\newtheorem*{remark}{Remark}
\newtheorem{proposition}{Proposition}[section]

\begin{document}

\title[Ridge detection and WSF]{Ridge detection for nonstationary multicomponent signals with time-varying wave-shape functions and its applications}

\author{Yan-Wei Su}
\address{Department of Applied Mathematics, National Yang Ming Chiao Tung University, Hsinchu, Taiwan}
\email{su311652001.sc11@nycu.edu.tw}

\author{Gi-Ren Liu}
\address{Department of Mathematics, National Cheng-Kung University, Tainan, Taiwan and National Center for Theoretical Sciences, National Taiwan University, Taipei, Taiwan}
\email{girenliu@gmail.com}

\author{Yuan-Chung Sheu}
\address{Department of Applied Mathematics, National Yang Ming Chiao Tung University, Hsinchu, Taiwan}
\email{sheu@math.nctu.edu.tw}

\author{Hau-Tieng Wu}
\address{Courant Institute of Mathematical Sciences, New York University, New York, NY, 10012 USA}
\email{hauwu@cims.nyu.edu}

\maketitle

\begin{abstract}
We introduce a novel ridge detection algorithm for time-frequency (TF) analysis, particularly tailored for intricate nonstationary time series encompassing multiple non-sinusoidal oscillatory components. The algorithm is rooted in the distinctive geometric patterns that emerge in the TF domain due to such non-sinusoidal oscillations. We term this method \textit{shape-adaptive mode decomposition-based multiple harmonic ridge detection} (\textsf{SAMD-MHRD}). A swift implementation is available when supplementary information is at hand. We demonstrate the practical utility of \textsf{SAMD-MHRD} through its application to a real-world challenge. We employ it to devise a cutting-edge walking activity detection algorithm, leveraging accelerometer signals from an inertial measurement unit across diverse body locations of a moving subject.
\end{abstract}

\section{Introduction}

Applying time-frequency (TF) analysis \cite{flandrin1998time} to study nonstationary time series with multiple oscillatory components has gained significant attention. Unlike traditional time or frequency domain analysis in time series \cite{brockwell2009time}, TF analysis follows a 'divide-and-conquer' approach. By assuming local signal stationarity, we extract and combine information to create the TF representation (TFR)\footnote{Other domains like time-scale, ambiguity, and time-lag are beyond our scope here.}. TFR is a function in the TF domain, a two-dimensional space with time on the x-axis and frequency on the y-axis.

Guided by this concept, various TF analysis algorithms emerged, ranging from linear methods like short-time Fourier transform (STFT) to nonlinear approaches like synchrosqueezing transform (SST) and its variations \cite{wu2020current}, along with others \cite{flandrin1998time}. An effective TF analysis algorithm accurately captures signal components across frequencies and times in the TF domain \cite{flandrin1998time}, enhancing encoded information use. The TFR, an extension of the Fourier transform, serves this purpose. Researchers manipulate the TFR to achieve tasks like signal decomposition, feature extraction, and denoising after obtaining it. Fig. \ref{overall flowchart} illustrates the signal processing flowchart. This paper focuses on TFR determined by STFT and SST to avoid distracting from the array of TF analysis algorithms.

An essential TFR manipulation step is {\em ridge detection} (RD), involving identifying ridges on the TFR that represent dominant signal traits.\footnote{RD can be discussed more generally \cite{hall1992ridge,steger1998unbiased}, where ridges are defined as curves on a surface. In this paper, we focus on TF analysis, where ridges are time-parametrized \cite[Section III]{carmona1999multiridge}.} Mathematically, a ridge is a sequence of local maxima of the TFR along the time-indexed frequency axis, often visualized as a curve on the TF domain. Under conditions like the adaptive harmonic model \cite{DaLuWu2011}, capturing multiple oscillatory components with slowly varying amplitude and frequency, ridges correspond to instantaneous frequencies (IF) of different components \cite{delprat1992asymptotic}, as shown in Fig. \ref{overall flowchart}. Ridge information enables tasks like component reconstruction \cite{DaLuWu2011} for signal decomposition \cite{colominas2021decomposing} and denoising \cite{Chen_Cheng_Wu:2014}. Further, phase and IF of each component can inform feature design for machine learning \cite{alian2022reconsider,alian2023amplitude}.

While the TF analysis-based signal processing flow in Fig. \ref{overall flowchart} has been utilized for real-world problems extensively, the pivotal step, RD, remains challenging. Literature has witnessed various endeavors to create efficient RD algorithms, encompassing path optimization with regularity constraints \cite{CHTridge1997,Chen_Cheng_Wu:2014}, its frequency modulation extension \cite{Iatsenko_McClintock_Stefanovska_2016,Colominas_Meignen_Pham_2020,Laurent_Meignen_2021}, Markov chain Monte Carlo (MCMC) \cite{carmona1999multiridge}, weighted ridge reconstruction (WRR) and self-paced ridge reconstruction (SPRR) \cite{Zhu_Zhang_Gao_Li_2019}, and a pseudo-Bayesian approach \cite{Legros_Fourer_2021}. TFR preprocessing via singular value decomposition (SVD) \cite{Ozkurt_Savaci_2005} or TF domain division using the reassignment vector \cite{Meignen_Gardner_Oberlin_2015,Laurent_Meignen_2021} before ridge determination is also viable. Noise impact is commonly mitigated via thresholding. Most RD algorithms, except \cite{carmona1999multiridge}, detect one ridge at a time from the entire TF domain or a portion, obtaining all ridges iteratively, often using a {\em peeling scheme}.

Despite successful applications, these algorithms have persistent limitations. All are founded on the assumption of sinusoidal oscillations and often require well-separated IFs; that is, no mode mixing. This assumption is challenged when oscillatory patterns lack sinusoidal traits, yielding multiple harmonics in the TFR. Refer to Fig. \ref{overall flowchart}(b)(e) for an instance involving a non-sinusoidal oscillation depicted in Fig. \ref{overall flowchart}(a) using second-order SST \cite{oberlin2015second}. This pattern is termed the {\em wave-shape function} (WSF) \cite{Wu:2013}. Real-world instances, particularly in biomedicine like the wrist accelerometer signal in Fig. \ref{fig2}, exhibit complex non-sinusoidal WSFs. Even with well-separated IFs of fundamental components, overlapping IFs of their harmonics can arise. See an example in Fig. \ref{two components demo}. This challenge, especially relevant to digital health, where diverse biomedical signals with intricate WSFs are prevalent, challenges the applicability of existing RD algorithms. Thus, a tailored RD algorithm is urgently needed.

In this paper, we introduce an innovative RD algorithm tailored for non-sinusoidal WSFs' influence on the TFR's geometric structure, and show its outperformance over existing methods. The algorithm's foundation rests on a fundamental theoretical insight. When WSFs exhibit non-sinusoidal traits, STFT and SST generate a unique geometric pattern in the TFR. This pattern, recurrent along the frequency axis at a fixed time point, is evident in Fig. \ref{overall flowchart}(b){(c)}, denoted by red arrows. Our approach involves fitting {\em multiple ridges} simultaneously to the TFR respecting this structure. When there are multiple oscillatory components, {\em shape-adaptive mode decomposition} (SAMD) \cite{colominas2021decomposing} is applied to replace the commonly applied peeling scheme.
We term it \textit{SAMD-based multiple harmonic RD} (\textsf{SAMD-MHRD}). Incorporating this geometric structure is akin to accessing more ridge information, akin to the principles in cepstrum \cite{oppenheim2004frequency} and de-shape STFT \cite{lin2016waveshape}. Fig. \ref{overall flowchart}(g) depicts \textsf{SAMD-MHRD} yielding four harmonic fits. \textsf{SAMD-MHRD} not only offers an accurate RD, but also enhances SAMD's decomposition capabilities. As a practical application, it aids accurate walking activity detection via accelerometers placed on various body areas, particularly the wrist.

\begin{figure*}[hbt!]
\centering
\includegraphics[trim=0 90 0 70,clip,width=1\textwidth]{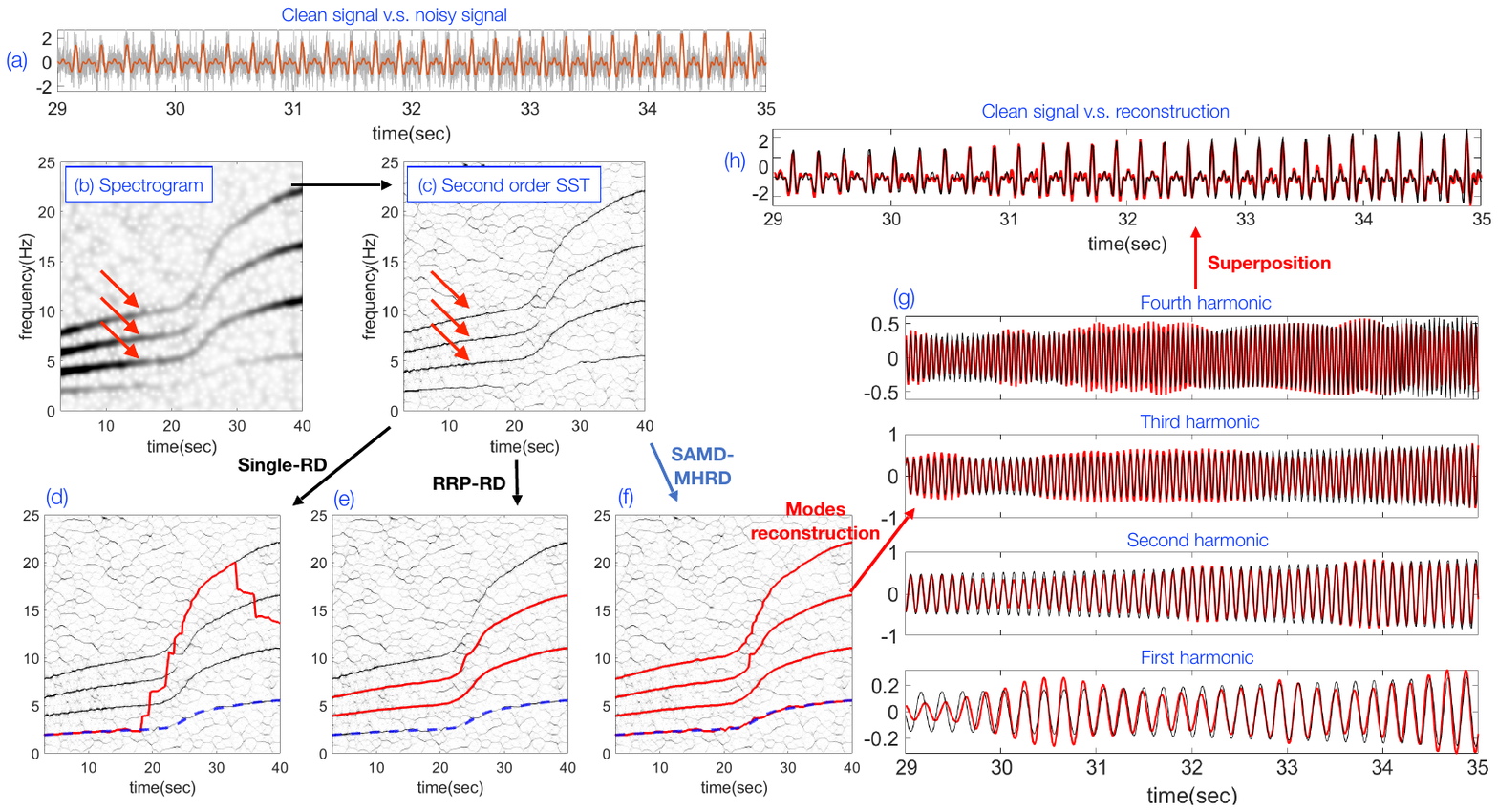}
\caption{The overall flowchart of analyzing nonstationary oscillatory time series and the challenge of RD.
(a) A portion of the simulated noisy signal $Y(t)$ with the signal-to-noise level of 5 dB is shown as the gray curve, with the clean signal superimposed as the red curve. The clean signal has one intrinsic mode type (IMT) function, whose fundamental component is weak. See \eqref{Model:equation2} for the precise definition.
(b) The spectrogram of $Y(t)$.
(c) The second order SST.
(d) The extracted lowest-frequency ridge from the TFR shown in (c) by the existing curve extraction algorithm, \textsf{Single-RD} \eqref{singleCurveExt:theory}, is superimposed as the red curve. The extracted ridge is biased by the ridges of the harmonics after the 19th second. The true instantaneous frequency (IF) is superimposed as the blue-dashed line.
(e) Same as (d), but by another existing RD algorithm, RRP-RD.
(f) Same as (d), but by our proposed curve extraction algorithm \textsf{SAMD-MHRD}.
(g) The decomposed harmonics of the IMT function are shown in red, and the true harmonics are shown in gray. Note that from the 29th to the 31st second, the fundamental component is less accurately recovered due to the low signal-to-noise ratio. It is the same for the fourth harmonic. 
(h) The summation of the decomposed harmonics is shown in red, and the true IMT function is shown in black.
\label{overall flowchart}}
\end{figure*}

The paper's structure is as follows: Section \ref{section model} introduces the mathematical model for nonstationary multicomponent signals with time-varying WSF. This section also summarizes the application of TF analysis tools in signal processing, discusses RD as a pivotal analysis step, and highlights existing challenges. Section \ref{section multiple curve ext} elaborates on the proposed \textsf{SAMD-MHRD} algorithm. Numerical examples are presented in Section \ref{section numerics}, while the real-world application is demonstrated in Section \ref{section real signals app}. Finally, Section \ref{section discussion conclusion} contains the discussion and conclusion.

\section{Mathematical models and analysis tools}\label{section model}

We begin by examining the mathematical model for non-stationary multicomponent signals with time-varying WSFs. Following that, we provide brief overviews of STFT and SST, illustrating their application in diverse signal processing tasks. Subsequently, we identify the constraints of current RD algorithms, which serve as the inspiration for this study.

\subsection{Mathematical model}
We consider the {\em adaptive non-harmonic model} (ANHM) to model a non-sinusoidal oscillatory signal.  
Fix a small constant $\epsilon>0$ and $L\in \mathbb{N}$. 
Assume the signal satisfies
\begin{equation}
\label{Model:equation}
Y_0(t) = \sum_{\ell=1}^L a_{\ell}(t)s_{\ell}(\phi_{\ell}(t))+T(t)+\Phi(t)\,,
\end{equation}
where $a_{\ell}(t)s_{\ell}(\phi_{\ell}(t))$ is called the $\ell$-th {\em intrinsic model type} (IMT) function that satisfies%$\chi_{I_\ell}$ is the indicator function and for each $\ell=1,\ldots,L$,
\begin{itemize}
\item[(C1)] $\phi_{\ell}(t)$ is a $C^2$ function called the {\em phase function} of the $\ell$-th IMT function. We assume $\phi'_\ell>0$;

\item[(C2)] $\phi_{\ell}'(t)$ is the instantaneous frequency (IF)  of the $\ell$-th IMT function %of the $\ell$-th interval 
so that $|\phi''_\ell(t)|\leq \epsilon \phi'_\ell(t)$ for all $t\in \mathbb{R}$ and $\inf_{t\in \mathbb{R}} \phi_1'(t)>\Delta$ for some $\Delta>0$. When $L>1$, we assume $\phi_{j+1}'(t)-\phi_j'(t)>\Delta>0$ for all $j=1,\ldots,L-1$ and $t\in \mathbb{R}$; 

\item[(C3)] $a_{\ell}(t)>0$ is a $C^1$ function denoting the {\em amplitude modulation (AM)} of the $\ell$-th IMT function so that $|a_\ell'(t)|\leq \epsilon \phi'_\ell(t)$ for all $t\in \mathbb{R}$;  

\item[(C4)] $s_{\ell}$ is a smooth $1$-period function with mean 0 and unit $L^2$ norm so that $|\hat{s}_\ell(1)|>0$, which is called the {\em wave-shape function} (WSF) of the $\ell$-th IMT function; 

\item[(C5)] $\Phi(\cdot)$ is a mean $0$ stationary random process in the wide sense with finite variance;

\item[(C6)] $T(t)$ is a smooth function so that {its Fourier transform, $\widehat T$,} is compactly supported in $[-\Delta,\Delta]$. 

\end{itemize}

The model's identifiability is discussed in \cite{Chen_Cheng_Wu:2014}. A typical example adhering to \eqref{Model:equation} is the accelerometer signal, with $L=1$, shown in Fig. \ref{fig2}. Clearly, the IMT function does not oscillate sinusoidally.
This model assumes {\em one} WSF for each IMT function. In essence, the $\ell$-th IMT function arises from stretching $s_\ell$ by $\phi_\ell$, subsequently scaled by $a_\ell$. Nevertheless, in practical instances, researchers have observed WSFs not to be fixed \cite{alian2023amplitude,eid2023using,wang2023arterial}. See Fig. \ref{fig2} for an example. It is evident that the oscillatory pattern varies over time. Notably, both STFT and SST reveal that the strength of the fundamental component (indicated by the red arrow) weakens, while the third, sixth, and seventh harmonics (indicated by purple arrows) strengthen after 225 seconds. As a result, \eqref{Model:equation} lacks appropriateness as a model.

\begin{figure}[htb!]
\centering
\includegraphics[trim=50 10 0 0, width=0.85\textwidth]{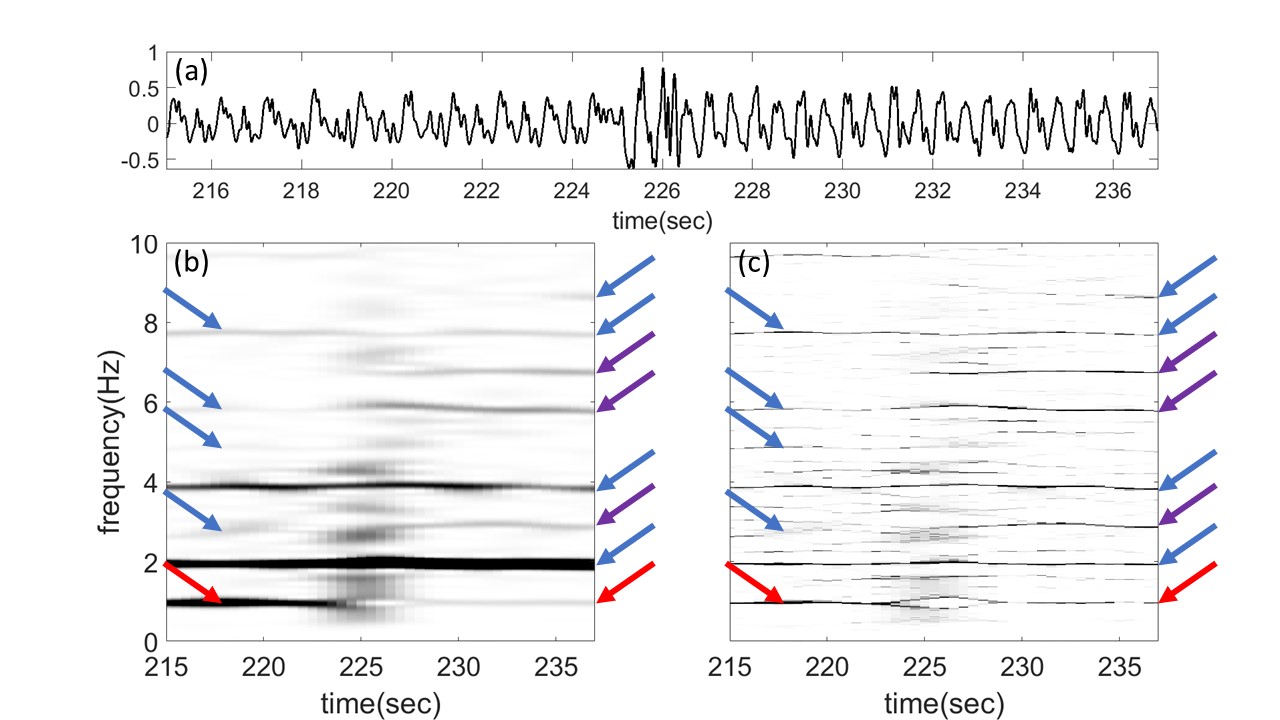}
\caption{Top row: an accelerometer signal. Bottom left: the spectrogram of the accelerometer signal. Bottom right: the SST of the accelerometer signal.  The fundamental component is indicated by the red arrow, and the harmonics are indicated by the blue arrows. \label{fig2}}
\end{figure}

To properly model such signals, we consider the generalization of \eqref{Model:equation} in \cite{lin2016waveshape}.
To motivate this generalization, assume $L=1$ to simplify the discussion. We have
\begin{align}
a_{1}(t)s_{1}(\phi_{1}(t))\label{Expansion f Fourier series}= \sum_{j=1}^\infty (\alpha_ja_1(t))\cos(2\pi j\phi_1(t)+\beta_j)\,,
\end{align}
where $\{\alpha_j\}\subset [0,\infty)$ and $\{\beta_j\}\subset[0,2\pi)$ are determined by Fourier coefficients of $s_1$. 
%We call $(\alpha_1a_1(t))\cos(2\pi \phi_1(t)+\beta_1)$ the {\em fundamental component} of $f(t)$, and for $j\geq 1$, we call $(\alpha_ja_1(t))\cos(2\pi j\phi_1(t)+\beta_j)$ the {\em $j$-th harmonic} of $f(t)$. 
The generalization idea in \cite{lin2016waveshape} involves permitting time-varying characteristics for each harmonic in \eqref{Expansion f Fourier series}, given specific conditions. Thus, consider
\begin{equation}
\label{Model:equation2}
Y(t) = \sum_{\ell=1}^L \sum_{j=1}^{D_\ell} a_{\ell,j}(t)\cos(2\pi\phi_{\ell,j}(t))+T(t)+\Phi(t)\,,
\end{equation} 
where $D_\ell\in \mathbb{N}$ is called the {\em harmonic order}, and
%\begin{itemize}
%\item [{(C7)}] $a_{\ell,l}\in C^1(\mathbb{R})\cap L^\infty(\mathbb{R})$ for $l=1,2,\ldots$ and $a_{\ell,1}(t)>0$ and $a_{\ell,l}(t)\geq 0$ for all $t$ and $l=2,3\ldots$; %$a_{\ell,l}(t) \leq c(l) B_{\ell,1}(t)$, for all  $t\in \mathbb{R}$ and $l = 1,2,\dots,D_\ell$, and with $ \{c(l)\}_{l = 1}^{\infty}$ a non-negative $\ell^{1}$ sequence. %Moreover, there exists $N\in \mathbb{N}$ so that $\sum_{l=N+1}^\infty B_l(t)\leq \epsilon' \sqrt{\sum_{l=1}^\infty B_l^2(t)}$ and $\sum_{l=N+1}^\infty lB_l(t)\leq D\sqrt{\sum_{l=1}^\infty B_l^2(t)}$ for some constant $D>0$.
%
%\item [{(C8)}] $\phi_{\ell,l}\in C^2(\mathbb{R})$ and $|\phi'_{\ell,l}(t) - l\phi'_{\ell,1}(t)| \leq \epsilon' \phi'_{\ell,1}(t)$, for all $t\in \mathbb{R}$,  $l = 1,\dots,\infty$ and a small constant $\epsilon'\geq 0$;
%
%\item [{(C9)}]  $|a'_{\ell,l}(t)| \leq \epsilon' c(l) \phi'_{\ell,1}(t)$ and $|\phi''_{\ell,l}(t)| \leq \epsilon' l \phi'_{\ell,1}(t)$ for all $t \in \mathbb{R}$, and $\sup_{l;\,B_l\neq 0}\|\phi''_l\|_\infty=M$ for some $M\geq 0$.
%
%\end{itemize}
%
\begin{enumerate}
\item[(C7)] $\phi_{\ell,1}$ behaves like $\phi_{\ell}$ in (C1)-(C2), $\phi_{\ell,j}\in C^2$, and $|j-\phi'_{\ell,j}(t)/\phi'_{\ell,1}(t)|\leq \epsilon'$ for $j=2,\ldots$ and a small constant $\epsilon'\geq 0$  for all $t\in \mathbb{R}$; 
\item[(C8)] $a_{\ell,1}(t)$ behaves like $a_{\ell}$ in (C3), $a_{\ell,j}\in C^1$ and $|a'_{\ell,l}(t)| \leq \epsilon' \phi'_{\ell,1}(t)$ for all $t\in \mathbb{R}$. %and $\{a_{\ell,j}(t)/a_{\ell,1}(t)\}_j$ is an $\ell^1$ sequence for all $t\in \mathbb{R}$.
\end{enumerate}
Conditions (C7) and (C8) capture the time-varying WSF, assuming slow variation. This model is a simplified version of that in \cite{lin2016waveshape} to facilitate discussion. For $j\geq 1$, $a_{\ell,j}(t)\cos(2\pi \phi_{\ell,j}(t))$ is called the {\em $j$-th harmonic} of the $\ell$-th IMT function, and $a_{\ell,1}(t)\cos(2\pi \phi_{\ell,1}(t))$ is the {\em fundamental component}. When $D_\ell=1$, the $\ell$-th IMT function oscillates sinusoidally. In this work, we assume the knowledge of $L$ and apply \cite{ruiz2022wave} to estimate $D_\ell$ to avoid the distraction. {Although there are tools when $D_\ell=1$ for all $\ell$ \cite{SucicSaulig2011,SauligPustelnik2013,Laurent_Meignen_2021,ruiz2022wave}, estimating $L$ in general is challenging. The slowly varying IF condition (C2) can be extended to a fast varying one \cite{kowalski2018convex}, which requires more technical details. To stay focused, we will discuss this generalization in future work.}

\subsection{Analysis by TF analysis tools---a quick summary}\label{section summary tfa}

Due to the non-stationary nature of ANHM, conventional time series methods may be limited. A common strategy to analyze such signals involves employing TF analysis algorithms \cite{flandrin1998time}. The core notion behind TF analysis is a ``divide-and-conquer'' approach. Locally assuming a ``stable'', ``time-invariant'', or ``stationary'' structure, we estimate this structure accurately in patches. By combining local estimates, we construct a TFR, which enables solving signal processing tasks like AM and IF estimation of each IMT function, signal decomposition into essential components (IMT functions and their harmonics), denoising, and dynamic feature extraction. This paper concentrates on STFT and SST, and we refer readers to \cite{flandrin1998time} for other TF analysis algorithms.

Below, we explain the overall idea. Mathematically, for a function $f$ satisfying mild conditions (e.g., a tempered distribution), the STFT is defined as
\begin{equation}
V^{(h)}_f(t,\xi):=\int f(x)h(x-t)e^{-i2\pi\xi(x-t)}dx\,,
\end{equation}
where $t\in \mathbb{R}$ is the time, {$\xi\in \mathbb{R}$} is the frequency, and $h$ is a chosen window that is smooth and decays sufficiently fast (e.g., the Gaussian function). The spectrogram is $|V^{(h)}_f(t,\xi)|^2$, yielding a function on {$\mathbb{R}^2$}. When $L>1$, the essential support of $\hat{h}$ should reside in $[-\Delta/2,\Delta/2]$ or the spectral interference of different IMT functions might happen.  
Generally, any TF analysis method's output is a function on $\mathbb{R}\times \mathbb{R}^+$, known as the TFR. We call $\mathbb{R}\times \mathbb{R}^+$ the {\em TF domain}. See Fig. \ref{overall flowchart}(b) for an example, where the input signal follows ANHM with $L=1$, where the fundamental component is weaker compared to its second harmonic and diminishes over time, as evident in the spectrogram. 
Nonetheless, STFT encounters the uncertainty principle \cite{ricaud2014survey}; that is, the TFR appears 'blurred' due to the chosen window. %In other words, local information capture is possible for studying dynamics, but comes with inherent uncertainty. 
To address this, SST \cite{DaLuWu2011}, a reassignment algorithm  \cite{auger1995improving} variant, and its generalizations like second-order SST \cite{oberlin2015second} were introduced. They utilize the phase information within the complex-valued TFR obtained through STFT to {sharpen} the STFT-based TFR. Let $S_f^{(h,[1])}:\mathbb{R}\times \mathbb{R}^+\to \mathbb{C}$ and $S_f^{(h,[2])}:\mathbb{R}\times \mathbb{R}^+\to \mathbb{C}$ denote the original SST \cite{DaLuWu2011} and second-order SST \cite{oberlin2015second} TFRs of function $f$, respectively.

To proceed, the conventional procedure entails extracting {\em ridges} connected to harmonic IFs. In contrast to STFT, a sharper TFR, as determined by SST, can enhance ridge estimation accuracy \cite{Meignen2017}. It has been theoretically established that when a signal adheres to ANHM with sinusoidal WSFs, the ridges closely approximate IFs of harmonics of all IMT functions \cite{delprat1992asymptotic,DaLuWu2011}.
Once ridges are accurately extracted, IMTs can be reconstructed through a reconstruction formula. Specifically, suppose the extracted ridge $c(t)$ approximates IF of $a_{\ell,j}(t)\cos(2\pi\phi_{\ell,j}(t))$ in \eqref{Model:equation2} for some $\ell$ and $j$. Then, if $\phi_{\ell,j}'(t)$ {\em does not intersect with and is away from IF of any other harmonics}, the following formula holds:
\begin{equation}
    \int_{|\xi-c(t)|<\Delta_q}S_f^{(h,[q])}(t,\xi)d\xi \approx a_{\ell,j}(t)e^{ 2\pi i\phi_{\ell,j}(t)}\,,
    \label{reconFormula}
\end{equation}
where $q=1,2$ and user-defined bandwidth $\Delta_q>0$. This formula yields { estimates of} $a_{\ell,j}(t)$ through absolute value extraction and $\phi_{\ell,j}(t)$ through phase unwrapping \cite{alian2022reconsider}. With these {estimates}, further signal processing tasks become feasible. Regrettably, the condition that $\phi_{\ell,j}'(t)$ must not intersect with or approach the IF of other {IMT function's} harmonics, as required by \eqref{reconFormula}, is excessively stringent and practically infeasible. Thus, an alternative approach becomes necessary. Various mode retrieval algorithms for STFT exist, such as ridge-based skeleton \cite{CHTridge1997}, integration approach \cite{Laurent_Meignen_2021}, basin attractor utilization \cite{Meignen_Gardner_Oberlin_2015,meignen2016adaptive,Laurent_Meignen_2021}, penalization with detected ridges \cite{carmona1999multiridge}, and others. However, these methods face similar limitations posed by non-sinusoidal WSFs.

\begin{remark}
It is worth noting that the phase can be estimated by cumulatively summing the estimated IF. However, this method is suboptimal due to accumulating discretization errors and the need for a separate initial phase estimation. This approach is only advisable when \eqref{reconFormula} or other reconstruction formulas cannot be applied.
\end{remark}

See Section \ref{section numerical implementation of TFA} for a discussion of the numerical implementation of STFT and SST. Below, we denote the discretized TFR of $f$, either determined by STFT or SST, as $\mathbf{R}\in \mathbb{C}^{N\times M}$, where $N$ is the number of the sampling points of the signal with the sampling period $\Delta_t$, and $M$ is the number of the frequency domain ticks, each bin spaced by $\frac{1}{2M\Delta_t}$.

\subsection{Ridge detection as a key step---existing algorithms}
The key step in the signal processing process outlined in Section \ref{section summary tfa} involves an accurate RD algorithm. 
RD algorithms can be broadly categorized into three groups. One involves detecting ridges individually and gathering all ridges through an iterative {\em peeling scheme} \cite{CHTridge1997,Chen_Cheng_Wu:2014,Colominas_Meignen_Pham_2020}. The second is the {\em multiridge scheme}, which detects multiple ridges concurrently \cite{carmona1999multiridge}. The third is the {\em preprocessing scheme}, where the TF domain is segmented to include one ridge per portion or the TFR is enhanced before applying either peeling or multiridge algorithms \cite{Laurent_Meignen_2021}. Given that our proposed algorithm adheres to the peeling scheme and the focus is handling WSFs, our focus remains on reviewing algorithms within this category.

In the peeling scheme, once we obtain a ridge $c_1:[N]\rightarrow[M]$, where $[N]:=\{1,\cdots,N\}$, from $\mathbf{R}_1:=\mathbf{R}$, we set iteratively the following masked TFR, $\ell=2,\ldots,L+1$:
\begin{align}
\label{peeling scheme formula}
\mathbf{R}_\ell(n,m)=\left\{
\begin{array}{ll}
0 & \mbox{if }(n,m)\in B_{\ell-1,n}\\
\mathbf{R}_{\ell-1}(n,m) &\mbox{otherwise}\,.
\end{array}
\right.\,,
\end{align}
where $B_{\ell-1,n}:=[c_{\ell-1}(n)-\eta_-(n),c_{\ell-1}(n)+\eta_+(n)]$, $c_{\ell-1}$ is the ridge extracted from $\mathbf{R}_{\ell-1}$, and $\eta_-(n)\geq0$ and $\eta_+(n)\geq0$ is the bandwidth selected by the user. Note that we assume the knowledge of $L$ in this work. 
Practically, the selection of $\eta_-(n)$ and $\eta_+(n)$ depends on the TFR and profoundly impacts RD performance.  Several methods exist for masking the TFR. For instance, one can opt for a constant frequency bandwidth (CFB) approach, where $\eta_-(n)$ and $\eta_+(n)$ remain constant for all $n$ \cite{Colominas_Meignen_Pham_2020}. Alternatively, the bandwidth can adapt to noise levels, known as the varying frequency bandwidth (VFB) \cite{Colominas_Meignen_Pham_2020}. Another option involves the modulation-based (MB) approach, which relies on estimating spectral spreading due to chirp \cite{Colominas_Meignen_Pham_2020}.

Numerous algorithms exist for fitting a ridge to the (masked) TFR $\mathbf{R}_\ell$ in each iteration. These algorithms share the common goal of approximating a curve within the TFR to capture maximum energy while meeting specific conditions. A frequently used method is to solve the following path optimization problem \cite{CHTridge1997,Chen_Cheng_Wu:2014}:
\begin{equation}
c^*=\argmax_{c:[N]\rightarrow[M]}
\sum_{j=1}^N
\big|\widetilde{\mathbf{R}}(j,c(j))
\big|
-
\lambda_1\sum_{j=1}^{N-1}\left|\Delta c(j)\right|^2\,,
    \label{singleCurveExt:theory}
\end{equation}
where $\Delta c\in [M]^{N-1}:=\{(m_1,\cdots,m_{N-1}):m_k\in[M],\,1\leq k\leq N-1\}$ so that $\Delta c(j):=c(j+1)-c(j)$ for $j=1,\ldots,N-1$, which is the numerical differentiation of the curve $c$, $\lambda_1>0$ is the penalty term constraining the regularity of the fit curve $c^*$, and $\widetilde{\mathbf{R}}(\ell,q)=\log\frac{|\mathbf{R}(\ell,q)|}{\sum_{i=1}^N\sum_{j=1}^M|\mathbf{R}(i,j)|}$ is a normalization of the matrix $\mathbf{R}$. 
In practice, various approaches to solving the optimization problem (\ref{singleCurveExt:theory}) exist, including the simulated annealing algorithm \cite{CHTridge1997}. However, we recommend the more efficient penalized forward-backward greedy algorithm, referred to as the \textit{single curve extraction algorithm} (\textsf{Single-RD}) \cite{CHTridge1997,Chen_Cheng_Wu:2014}. Note that we could further enhance regularity by considering second-order numerical differentiation in \eqref{singleCurveExt:theory}, or even incorporating available noise information \cite{CHTridge1997}. For simplicity, we omit these aspects from the current discussion.%\footnote{A more general parametrization of the ridge with a penalization called snake penalization could also be considered \cite{CHTridge1997}, but this approach not suitable for our setup so we do not consider it here.} 

Several other algorithms proposed in \cite{Iatsenko_McClintock_Stefanovska_2016,Colominas_Meignen_Pham_2020,Laurent_Meignen_2021} could be viewed as solving variations of \eqref{singleCurveExt:theory}. For example, the frequency modulation (FM) approach proposed in  \cite{Colominas_Meignen_Pham_2020} solves
$c^*=\argmax_{c:[N]\rightarrow[M]}
\sum_{j=1}^N
\left|\mathbf{R}(j,c(j))
\right|^2$ 
such that $|\Delta c(j)-\frac{K}{N^2}\hat{q}(j,c(j))|\leq C$ for some $C\geq 0$, where $\hat{q}$ is the chirp rate estimate \cite[eq. (6)]{Colominas_Meignen_Pham_2020}. This optimization problem could be solved by, for example, the recurring principle \cite[eq. (11)]{Laurent_Meignen_2021}, leading to the \textsf{FM-RD} algorithm. In \cite{Laurent_Meignen_2021}, to handle substantial noise, the concept of a relevant ridge portion (RRP) is introduced, giving rise to the \textsf{RRP-RD} algorithm. RRPs correspond to the desired ridges and are employed to segment the TF domain into specific structures known as basins of attraction \cite{meignen2016adaptive}. The spline least-square approximation is applied to RRPs for RD.

As an example, consider the simplest scenario with only one IMT function ($L=1$ in \eqref{Model:equation2}) that oscillates sinusoidally ($a_{1,j}=0$ for all $j\geq 2$). In this case, the ridge $c^*$ determined by any of these RD algorithms provides an accurate estimate of the associated fundamental component's IF \cite{DaLuWu2011}, followed by other signal processing processes. Nonetheless, this simplest scenario is rarely encountered in practice.

\subsection{Challenges of existing approaches}

To illustrate the challenges posed by existing RD algorithms, consider scenarios closer to real-world situations. First, when $L>1$ in \eqref{Model:equation2} and all IMT functions oscillate sinusoidally, we can apply any of the aforementioned RD algorithms to estimate the IFs associated with all IMT functions, assuming the IF separation condition in (C2) holds. Second, when $L=1$ in \eqref{Model:equation2} but the WSF significantly deviates from being sinusoidal, we can apply the above RD algorithms based on (C7) to estimate all ridges linked to the harmonics. The ridge associated with the lowest IF can then be assumed as the fundamental component's IF. In these more realistic scenarios, successfully applying existing RD algorithms becomes more challenging. An example in Fig. \ref{overall flowchart}(f) demonstrates this complexity, where the strength of the fundamental component is weak in the middle while the strength of the second harmonic dominates. Such shifts in strengths can easily confuse existing RD algorithms.

In real-world scenarios, more intricate situations often challenge the aforementioned RD algorithms. When $L>1$ in \eqref{Model:equation2} and the WSFs deviate significantly from sinusoidal patterns, the TFR can become {complex} with multiple curves {representing} different harmonics of distinct IMT functions. This complexity involves interference {among} harmonics from different IMT functions and can lead to mode mixing. For instance, as shown in Fig. \ref{two components demo}, two IMT functions exhibit non-sinusoidal WSFs. Here, the second harmonic of the first IMT function (red arrows) intersects with the fundamental component of the second IMT function (blue arrows), making it challenging to accurately determine the IF of the fundamental component for each IMT function. 
Given that numerous signal processing tasks, such as SAMD \cite{colominas2021decomposing}, rely on {accurately} acquiring the phase of each IMT function and considering the prevalence of such signals, particularly in biomedical {applications}, there is a compelling need for a new RD algorithm.

\begin{figure}[hbt!]
\centering
\includegraphics[trim=4 20 4 55, clip, width=1\textwidth]{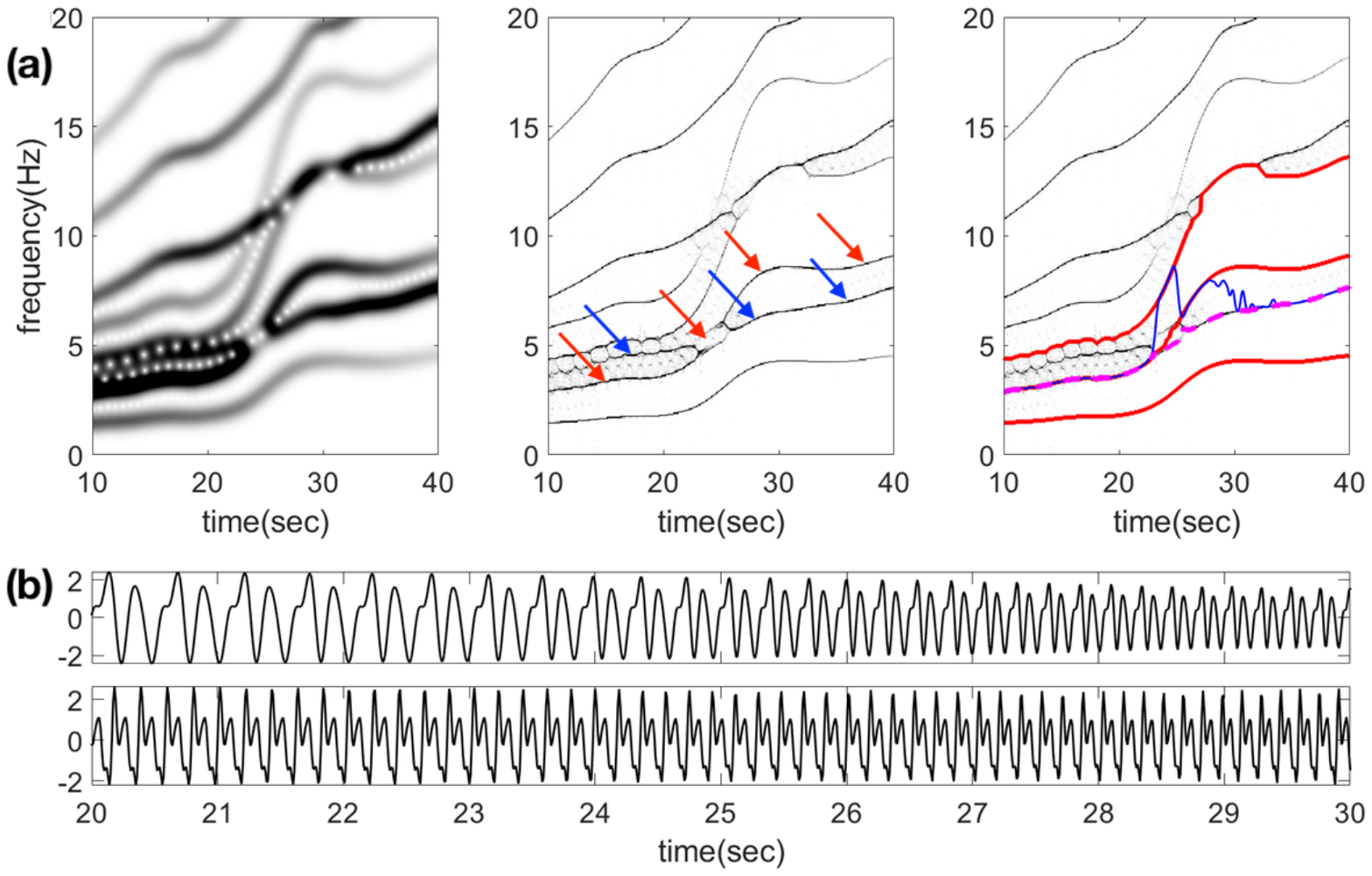}
\caption{Illustration of \textsf{Single-RD}, \textsf{FM-RD} and \textsf{RRP-RD} on a two-oscillatory components signal.
(a)-Left: Spectrogram.
We can see two curves that are associated with the fundamental components of the two IMT functions.
(a)-Middle: The second-order SST.
(a)-Right: The detection result using \textsf{SAMD-MHRD} (resp. \textsf{FM-RD} and \textsf{RRP-RD}) is superimposed on the 2nd-order SST as the red (resp. purple and blue) curves.
(b) A portion of the first and second IMT functions is shown for visual inspection.}
\label{two components demo}
\end{figure}

\section{Proposed ridge detection algorithm}\label{section multiple curve ext}

To address the aforementioned challenge, we introduce a novel RD algorithm within a peeling scheme framework, termed \textit{SAMD-based multiple harmonic RD} (\textsf{SAMD-MHRD}). This algorithm consists of two main iterative steps. First, it involves fitting ridges for multiple harmonics of a single IMT function, referred to as the {\em multiple harmonic RD} (\textsf{MHRD}) component. Subsequently, the algorithm progresses by iteratively ``peeling off ridges'' linked to the identified harmonics of the IMT function, constituting the \textsf{SAMD} phase. These steps are repeated until all IMT functions have been processed.   It is important to acknowledge that  each of these steps possesses its own interest and can be employed independently in alternative scenarios. Algorithm \ref{alg:improved-SAMD} provides a comprehensive summary of the \textsf{SAMD-MHRD} procedure. Below, we elaborate on these two steps and the iterative procedure. For the purpose of reproducibility, the Matlab implementation of the proposed algorithm and codes to regenerate results in Section \ref{section numerics} are available at  \url{https://github.com/yanweiSu/ridges-detection}.

Before detailing the algorithm's specifics, we highlight three innovations in \textsf{SAMD-MHRD}. First, we utilize the geometric structure of the TFR. To illustrate the idea, suppose $L=1$ and $\epsilon'=0$ in (C7). Then the STFT of $\sum_{j=1}^\infty a_{1,j}(t)\cos(2\pi j\phi_{1,1}(t))$ can be approximated by ${\frac{1}{2}}\sum_{j=1}^\infty a_{1,j}(t)\hat{h}(\xi-j\phi'_{1,1}(t))e^{2\pi i j\phi_{1,1}(t)}$ \cite{DaLuWu2011}; that is, at a fixed time $t$, the ridges associated with different harmonics are ``periodic'' in the frequency axis with the period $\phi'_{1,1}(t)$. This underlying periodic structure serves as the fundamental geometric element driving the design of \textsf{SAMD-MHRD}. Secondly, unlike the masking approach in \eqref{peeling scheme formula}, \textsf{SAMD-MHRD} employs the SAMD technique for ridge peeling. The masking technique can be limited in handling non-sinusoidal WSFs due to potential harmonic IF overlaps between different IMT functions. This can lead to error accumulation during iterations. In contrast, SAMD-based peeling effectively addresses these challenges and eliminates the need to determine optimal masking bandwidth. Third, \textsf{SAMD-MHRD} not only identifies ridges for all IMT functions but also decomposes the IMT functions themselves. Additionally, while SAMD aids ridge peeling in each iteration, \textsf{SAMD-MHRD} algorithm concurrently achieves signal decomposition like the original SAMD.

\subsection{Step 1: fit multiple ridges for harmonics (\textsf{MHRD})}
Given an input signal $f_1=f$ that satisfies \eqref{Model:equation2} with $L\geq 1$, let $\mathbf{R}_1\in \mathbb{C}^{N\times M}$ represent the discretized TFR of $f_1$, obtained from STFT or SST. Choose $K\geq 1$ as the intended number of harmonics to extract {\em simultaneously}. Using the notations from \eqref{singleCurveExt:theory}, we fit {$K$} ridges for the harmonics of {\em an} IMT function by solving the optimization problem:
%\newline
%\scriptsize
\begin{align}
\mathbf{c}^*=&\,
\underset{\mathbf{c}:[N]\rightarrow[M]^K}{\argmax}
\sum_{k=1}^K
\Big[
\sum_{j=1}^N\left|\widetilde{\mathbf{R}}_1\left(j, e_k^{\top}\mathbf{c}(j)\right)\right|\label{multiCurveFormula:generic}
\\&-\lambda_k\sum_{j=2}^N\left|e_k^{\top}\left(\mathbf{c}(j)-\mathbf{c}(j-1)\right)\right|^2-\mu_{k}\sum_{j=1}^N \left|e_k^{\top}\mathbf{c}(j)-ke_1^{\top}\mathbf{c}(j)\right|^2\Big]\nonumber\,,
\end{align}%\normalsize
where $e_k\in \mathbb{R}^K$, $k=1,\ldots,K$, is a unit vector with $e_k(k)=1$, $\lambda_k\geq 0$ is the {\em smoothness} penalty as in (\ref{singleCurveExt:theory}), and $\mu_k\geq 0$ stands for the \textit{similarity} penalty that captures the time-varying wave-shape condition stated in (C7). Note that $\mu_{k}>0$ forces the estimated IF of the $k$-th harmonic, $e_k^{\top}\mathbf{c}$ to satisfy $e_k^{\top}\mathbf{c}(\ell)\approx ke_1^{\top}\mathbf{c}(\ell)$ at time $\ell\Delta_t$, and hence the geometric structure is respected. The output $\mathbf{c}^*\in [M]^{K\times N}$ represents the IFs of the first $K$ harmonics of an IMT function of the input signal $f$.

{We simplify \eqref{multiCurveFormula:generic} by introducing a set}
\begin{align}
\mathbf{S}^{(K)}&\,(\beta)=
\big\{\mathbf{c}:[N]\to
[M]^{K}:
\label{multiCurveExt: feasible region}\\
&
|e_k^\top\mathbf{c}(j)-ke_1^\top\mathbf{c}(j)|\leq \beta e_1^\top\mathbf{c}(j)\;\forall k\in[K]\big\}\,,\nonumber
\end{align}
where $\beta\geq 0$ is the chosen ``bandwidth'' that echos the similarity penalty {$\mu_k$} in (\ref{multiCurveFormula:generic}), so that (\ref{multiCurveFormula:generic}) becomes 
\begin{equation}
 \mathbf{c}^*=\underset{\mathbf{c}\in\mathbf{S}^{(K)}(\beta)}{\argmax}
\sum_{k=1}^K
\bigg[
\sum_{j=1}^N \left|\widetilde{\mathbf{R}}_1\left(j, e_k^{\top}\mathbf{c}(j)\right)\right| -\lambda_k\sum_{j=2}^N\left|e_k^\top(\mathbf{c}(j)-\mathbf{c}(j-1))\right|^2\bigg]\,.
\label{multiCurveFormula:practice}
\end{equation}
Solving (\ref{multiCurveFormula:practice}) directly can be inefficient. {We recommend a {\em segmentwise} approach.
This method leverages the slowly-varying IF's uniform continuity by approximating the curve with a finite sequence of ``thin rectangles'' $\{[t_{q-1},\,t_{q}]\times[L_q,U_q]\}_{q=1}^Q$. First, estimate the fundamental IF $\hat c$ by running the algorithm on $[t_0,\,t_1]$ over the whole frequency domain. Then, for each segment $[t_{q},\,t_{q+1}]$, run the algorithm over a frequency band $[-B_q,B_q]+\hat c(t_q)$, where $B_q>0$ is chosen properly to ensure the box $[t_{q},\,t_{q+1}]\times[\hat c(t_q)-B_q,\hat c(t_q)+B_q]$ covers the ridge. This process determine the fundamental IF over $[t_q,t_{q+1}]$.
Note that a smaller $B_q$ reduces the search space and improves efficiency, but it must not be too small to avoid missing the ridge. We suggest using $B_q=1$ and $t_q-t_{q-1}=1$ for all $q$, based on the slowly-varying IF assumption to ensure the ridge is covered.} This section of the algorithm is referred to as \textsf{MHRD} and is summarized in Algorithm \ref{alg1:MultiCurveExt}.

\begin{algorithm}[hbt!]
\caption{\textsf{SAMD-MHRD} algorithm.}
\label{alg:improved-SAMD}
\small
    \begin{algorithmic}[1]
    \STATE \textbf{Input:}  $\mathbf{f}_0$: the signal, $L\in \mathbb{N}$: the number of IMT functions,
    $I\in \mathbb{N}$: the iteration times, $\{t_q\}_{q=0}^Q$: a strictly increasing sequence of the time segment points with $t_Q:=N$ and $t_0:=0$,
    $\{B_q\}_{q=1}^Q$: the width of the fundamental band in each time interval.
    \STATE $\mathbf{f}\leftarrow\mathbf{f}_0$
        \FOR{$i=1:I$}
        \FOR{$\ell=1:L$}
        \STATE Run the second order STFT-SST: $\mathbf{R}\leftarrow \mathbf{S}^{[2]}_\mathbf{f}$.
        \STATE $\mathbf{c}_\ell\leftarrow\mathsf{MHRD}(\mathbf{R},K,\{\lambda_k\}_{k=1}^K,\beta,\{t_q\},\{B_q\})$ (See Algorithm \ref{alg1:MultiCurveExt}).
        \ENDFOR
        \STATE Reorder $\{e_1^\top\mathbf{c}_1,\cdots,e_1^\top\mathbf{c}_L\}$ from low value to high value as $\{e_1^\top\mathbf{c}_{(1)},\cdots,e_1^\top\mathbf{c}_{(L)}\}$, {where $e_1$ is defined in \eqref{multiCurveFormula:generic}}.
        \FOR{$\ell=1:L$}
        \STATE $\mathbf{x}_{\ell,1}^{[i]}\leftarrow\text{Reconstruction with }e_1^\top\mathbf{c}_{(\ell)}$ by (\ref{reconFormula}).
        \STATE $\mathbf{x}^{[i]}_\ell\leftarrow$ SAMD with the estimated fundamental mode $\mathbf{x}_{\ell,1}^{[i]}$.
        \STATE $\mathbf{f}\leftarrow\mathbf{f}_0-\mathbf{x}^{[i]}_\ell$.
        \ENDFOR
        \ENDFOR
        \STATE \textbf{Output:} $\{\mathbf{x}_1^{[i]}\}_{i=1}^I,\ldots,\{\mathbf{x}_L^{[i]}\}_{i=1}^I$.
    \end{algorithmic}
\end{algorithm}

\begin{algorithm}[hbt!]
\caption{Multiple harmonic ridge detection algorithm. \textsf{MHRD}$\big(\mathbf{R},K,\{\lambda_k\}_{k=1}^K,\beta,\{t_q\}_{q=1}^Q,\{B_q\}_{q=2}^Q\big)$}
\label{alg1:MultiCurveExt}
\begin{algorithmic}[1]
\small
    \STATE \textbf{Input:}
    $\mathbf{R}\in\mathbb{C}^{N\times M}$: the TFR,
    $K$: the number of desired harmonics,
    $\lambda=(\lambda_1,\cdots,\lambda_K)$: the smoothness penalty in (\ref{multiCurveFormula:practice}),
    $\beta$: the ``bandwidth'' in (\ref{multiCurveExt: feasible region}).

    \STATE  {Assign $t_0\leftarrow 0$.}  For each $q\in[Q]$, define a submatrix $\mathbf{R}_q\in\mathbb{C}^{(t_q-t_{q-1})\times M}$ of $\mathbf{R}$ as $\mathbf{R}_q(\ell,m)=\mathbf{R}(\ell+t_{q-1},m)$ for all $\ell\in\{1,\cdots,t_q-t_{q-1}\}$ and $m\in[M]$.

    \STATE Solve (\ref{multiCurveFormula:practice}) over $\mathbf{R}_1$ and get $\hat{\mathbf{c}}_1\in[M]^{t_1\times K}$.
    Assign $\hat{\mathbf{c}}\leftarrow \hat{\mathbf{c}}_1$.

    \FOR{ $q=2:Q$ }
        \STATE $\mathbf{S}^{(K)}(\beta)\leftarrow\mathbf{S}^{(K)}(\beta)\cap\{\mathbf{c}\in[M]^{(t_q-t_{q-1})\times K}:\mathbf{c}(\cdot,1)\in[-B_q,B_q]+\hat{\mathbf{c}}(t_{q-1},1)\text{ and }\mathbf{c}(1,\cdot)=\hat{\mathbf{c}}(t_{q-1},\cdot)\}$.
    \STATE Solve (\ref{multiCurveFormula:practice}) over $\mathbf{R}_q$ and ${\bf S}^{(K)}$ and get $\hat{\mathbf{c}}_q\in[M]^{(t_q-t_{q-1})\times K}$. Assign $\hat{\mathbf{c}}\leftarrow[\hat{\mathbf{c}}^\top\;\hat{\mathbf{c}}_q^\top]^\top$.
    \ENDFOR
    \STATE \textbf{Output: }$\widehat{\mathbf{c}}\in[M]^{N\times K}$, where the estimated IF of the $k$-th harmonic is given by the $k$-th column of $\hat{\mathbf{c}}$.
\end{algorithmic}
\end{algorithm}

\subsection{Step 2: Peel off ridges for harmonics by \textsf{SAMD}}

After identifying one IMT function's harmonic ridges using \textsf{MHRD}, we employ the SAMD algorithm \cite{colominas2021decomposing} to extract the corresponding IMT function {($\mathbf{x}^{[1]}_1$ in Algorithm \ref{alg:improved-SAMD})} related to the detected ridges with its fundamental component's IF. { In short, SAMD is an optimization-based fitting algorithm that uses the estimated fundamental component of an IMT function ($\mathbf{x}_{\ell,1}^{[i]}$ in Algorithm \ref{alg:improved-SAMD} using (\ref{reconFormula}) with $e_1^\top\mathbf{c}_{(\ell)}$) as inputs. The loss function accounts for the time-varying WSF, and its minimizer provides an estimate of the IMT function. See \cite{colominas2021decomposing} for details.}
 
Subtracting this extracted IMT function from the input signal $f_1$ results in a new signal denoted as $f_2$, containing $(L-1)$ IMT functions. The discretized TFR of $f_2$, determined by STFT or SST, is represented as $\mathbf{R}_2\in \mathbb{C}^{N\times M}$. Note that ridges related to the extracted IMT function do not exist in $\mathbf{R}_2$. This step thus serves as a replacement for \eqref{peeling scheme formula}.

Having obtained the initial IMT function and $\mathbf{R}_2$, we can proceed to iterate Step 1 (\textsf{MHRD}) on $f_2$ and $\mathbf{R}_2$ to extract the harmonic ridges of one of the remaining IMT functions. Then, Step 2 is once again executed to extract the associated IMT function using the SAMD algorithm. This iterative process is repeated for a total of $L$ cycles, resulting in the complete set of harmonic ridges for all IMT functions. These initial ridge estimates, denoted as $c_1^*,\ldots,c_L^*$, are ordered based on the fundamental components' IFs, from low to high. 

In practice, it is not guaranteed which IMT function will be extracted in Steps 1 and 2. It is possible that the IF and phase of its associated fundamental component is affected by the harmonics of other IMT functions with lower IF.
Thus, with the initial estimate $c_1^*,\ldots,c_L^*$, we repeat Steps 1 and 2 on $f_1$ and $\mathbf{R}_1$ by replacing $\mathbf{S}^{(K)}(\beta)$ in \eqref{multiCurveFormula:practice} by
\begin{align}
\bar{\mathbf{S}}^{(K)}&\,(\beta)=
\big\{\mathbf{c}:[N]\to
[M]^{K-1}:
\label{multiCurveExt: feasible region2}\\
&
|e_{k-1}^\top\mathbf{c}(j)-kc^*_1(j)|\leq \beta c^*_1(j)\;\forall k\in\{2,\ldots,K\}\big\}\nonumber
\end{align}
and estimating the phase via \eqref{reconFormula}.
This ridge estimate is less influenced by harmonics from other IMT functions due to the IF separation condition (C2). We then repeat Steps 1 and 2 on $f_2$ and $\mathbf{R}_2$ with $\mathbf{S}^{(K)}$ modified in the same way as \eqref{multiCurveExt: feasible region2} by considering $c_2^*$ . We continue this process iteratively with $c_3^*, c_4^*,\ldots$ and so on until ridges of all $L$ IMT functions are extracted, denoted as $\bar c_1^*,\ldots,\bar c_L^*$.

\subsection{Iteration}

While it is valid to conclude the process after these iterations, it is possible to enhance the results through an additional iteration of the $L$ cycles with $\bar c_1^*,\ldots,\bar c_L^*$. This reiteration step has the potential to yield improved outcomes, and we recommend considering this option for $I\geq 1$ times. It is worth noting that the output includes the decomposition of all IMT functions from the input signal $f$.

\subsection{Speed up \textsf{SAMD-MHRD} with auxiliary information}

Solving the optimization problem (\ref{multiCurveFormula:generic}) can be time-consuming, especially when $K$ is large and $\beta$ is wide. {However, if the phase of the fundamental component of the targeted IMT function is known or can be well estimated in advance, the algorithm can be accelerated.} 
Let $\widetilde{\mathbf{R}}$, $\lambda_k$ and $\mu_{k}$ be the same as \eqref{multiCurveFormula:generic}, and $c_1:[N]\rightarrow[M]$ be the {\em known} IF of the fundamental component; {that is, the extra auxiliary information}. 
In this case, we only need to estimate the higher-order harmonics, and hence (\ref{multiCurveFormula:generic}) is reduced to
\begin{align}
\mathbf{c}^*
%\argmax_{\mathbf{c}:[N]\rightarrow[M]^{K-1}}
%\sum_{k=1}^{K-1}
%\bigg[
%\sum_{\ell=1}^N\left|\widetilde{\mathbf{R}}\left(\ell, e_k^{\top}\mathbf{c}(\ell)\right)\right|^2\nonumber
%\\&\quad-
%\lambda_k\sum_{\ell=2}^N\left|e_k^{\top}\left(\mathbf{c}(\ell)-\mathbf{c}(\ell-1)\right)\right|^2\nonumber
%\\&\quad-\mu_{k+1,\ell}\left(\left|e_k^{\top}\mathbf{c}(\ell)-(k+1)c_1(\ell)\right|\right)\bigg]\nonumber
 =\argmax_{\mathbf{c}:[N]\rightarrow[M]^K}
\sum_{k=1}^{K-1}\mathcal{L}^{(k+1)}_{c_1}(e_k^{\top}\mathbf{c})\,,
\label{multiCurveExt:with Fundamental}
\end{align}
where $\mathcal{L}_{c_1}^{(k)}(\mathbf{f}):=\sum_{j=1}^N\left|\tilde{\mathbf{R}}\left(j,\mathbf{f}(j)\right)\right|-\lambda_k\sum_{j=2}^N|\mathbf{f}(j)-\mathbf{f}(j-1)|^2-\mu_{k}\sum_{\ell=1}^N|\mathbf{f}(j)-kc_1(j)|^2$ for $k\in\{2,\cdots,K\}$ and $\mathbf{f}:[N]\rightarrow[M]$.
The following proposition shows that the optimization problem \eqref{multiCurveExt:with Fundamental} can be reduced to several single-valued curve optimization problems{, where the proof can be found in Section \ref{SIsection proof of prop: speed up algorithm}.}
\begin{proposition}\label{prop: speed up algorithm}
{\it Fix $K\geq 2$. The curve $\mathbf{h}_k:[N]\rightarrow[M]$ satisfying
\begin{equation}
\mathbf{h}_k
=
\argmax_{h:[N]\rightarrow[M]}\mathcal{L}_{c_1}^{(k+1)}(h)\,,
\end{equation}
where $1\leq k\leq K-1$, is one of the rows of $\mathbf{c}^*$ in (\ref{multiCurveExt:with Fundamental}).}
\end{proposition}

This proposition allows us to estimate harmonics' IFs by solving a \textit{modified single curve extraction} program with respect to $c_1$ for each $2\leq k\leq K$; that is, 
\begin{equation}
\mathbf{h}_k^*=\argmax_{h:[N]\rightarrow[M]}\mathcal{L}_{c_1}^{(k)}(h), 
\end{equation}
instead of solving (\ref{multiCurveExt:with Fundamental}) directly. This reduces the complexity of the original problem. 
{This auxiliary information is often available in many applications. For example, when analyzing photoplethysmogram (PPG) and electrocardiogram (ECG) signals, we can use the cardiac phase from the ECG as an accurate phase estimate for the cardiac component in the PPG signal. This allows us to speed up \textsf{SAMD-MHRD} when analyzing the cardiac component in the PPG signal.}

\subsection{Parameters selection}
\label{section parameter selection}

To optimize the performance of \textsf{SAMD-MHRD} for a given $K$, selecting suitable penalties $\lambda_1, \ldots, \lambda_K$ is crucial. To address this, we propose a simple grid search approach with the following rationale. Assuming that the algorithm accurately detects the ridges corresponding to the true IFs, we can consider masking the extracted curves on the TFR, similar to \eqref{peeling scheme formula}, by appropriately choosing $\eta_-$ and $\eta_+$. {To simplify the discussion, we apply the CFB approach \cite{Colominas_Meignen_Pham_2020} below, while VFB and MB can also be considered.} The resulting TFR should exhibit reduced signal components. This insight prompts us to employ the $\alpha$-R\'enyi entropy \cite{baraniuk2001measuring,Sheu_Hsu_Chou_Wu:2017}, where $\alpha>0$, as a metric to quantify and compare the energy distribution along the frequency axis of both the masked and original TFRs.
%Recall that for $\alpha>0$,  the $\alpha$-R\'enyi entropy, denoted as $R_\alpha$, of a non-zero and non-negative function $p$, is defined as
%\begin{equation*}
%R_\alpha(p)=\frac{1}{1-\alpha}\log_2\left(\frac{\|p\|_{2\alpha}}{\|p\|_\alpha}\right)^{2\alpha}.
%\end{equation*}

Consider the parameter set 
\begin{align}
\mathcal{S}_{\delta_\lambda,\Lambda}:=\big\{&(\lambda_1,\cdots,\lambda_K):\lambda_k=
(1-(k-1)\delta_\lambda)\lambda_1\nonumber
\\& \text{ for }  k=1,\ldots,K\,,\lambda_1\in\Lambda
\big\}\,,
\end{align}
where  $\delta_\lambda>0$ is small, $K$ satisfies $1-(K-1)\delta_\lambda>0$, and $\Lambda\subset(0,\infty)$ is a finite set. Note that while $\mathcal{S}_{\delta_\lambda,\Lambda}$ looks like a high-dimensional space, it is effectively a one-dimensional space. In this context, the ordering $\lambda_k<\lambda_\ell$ for $k>\ell$ adheres to the \textit{slowly varying} condition $|\phi_j^{\prime\prime}|\leq\epsilon j\phi_1^\prime$ for $j\geq 1$.
For the parameter $\beta$ required to satisfy the WSF condition in $\mathbf{S}^{(K)}(\beta)$ (\ref{multiCurveExt: feasible region}), we consider a finite subset $M\subset(0,2^{-1})$.

Take the TFR $\mathbf{R}$.
For each $\lambda_1\in\Lambda$ and $\beta\in M$, extract the curves by \textsf{SAMD-MHRD} with $(\lambda_1,\cdots,\lambda_K)\in S_{\delta_\lambda,\Lambda}$ and $\beta$, and denote the result to be $\mathbf{c}^*(\lambda_1,\beta)$.
Then, ``mask'' $\mathbf{R}$ by the ridges $\mathbf{c}^*(\lambda_1,\beta)$ following \eqref{peeling scheme formula}, and denote the masked TFR as $\mathbf{R}_{(\lambda_1,\beta),\Delta}$, where the masking size $\Delta>0$ respects the one used in the reconstruction formula (\ref{reconFormula}). Define a sequence $(q^{(\lambda_1,\beta)}(1),\cdots,q^{(\lambda_1,\beta)}(N))$ by $q^{(\lambda_1,\beta)}(\ell):=R_\alpha(|\mathbf{R}_{(\lambda_1,\beta),\eta}(\ell,\cdot)|)$, $\ell\in[N]$.
Finally, we choose
\begin{equation}
(\lambda_1^*, \beta^*)=\argmax_{(\lambda_1,\beta)\in\Lambda\times M}
\sum_{\ell=1}^N
\frac{q^{(\lambda_1,\beta)}(\ell)}{N}
\end{equation}
as the optimal $(\lambda_1,\beta)$ in $\Lambda\times M$.
In practice, $\Lambda$ is chosen to range from $10^{-1}$ to $10^1$ and $M$ is chosen to range from $2^{-6}$ to $2^{-1}$, both with a uniform grid in the log scale.

\section{Numerical simulation}\label{section numerics}

\subsection{Single IMT function ($L=1$)}\label{section simulation L=1}
In this first example, we compare different RD algorithms on signals with complicated WSF dynamics. Denote the smoothed standard Brownian motion $X_1(t)=W\ast K_{B}(t)$, where $W$ is the standard Brownian motion, $\ast$ indicates convolution, and $K_{B}$ is the Gaussian function with the standard deviation $B=20$ seconds.
Over $[0,50]$, set 
$$
A_1(t)=e^{-\left(\frac{t-10}{30}\right)^2}\left(3\int_0^t\frac{|X_1(s)|}{\|X_1\|_{L^\infty[0,50]}}ds+\frac{5}{2}\right)
$$
and 
$$
\phi_\ell(t)=\ell[\phi_0(t)+\int_0^t\frac{X_2(s)}{\|X_2\|_{L^\infty[0,50]}}ds]+U_\ell\ast K_{B'}(t)\,,
$$ 
where $X_2$ is an independent and identical copy of $X_1$, $U_\ell$ are independent white Gaussian process with mean 0 and variance $0.1$, $B'=5$ seconds, and
\begin{equation}
\phi_0(t)=
0.97t+\frac{t^{1.9}}{38}
+
\frac{3}{2}
\int^t_0
\int_0^s\frac{e^{-\left(\frac{u-25}{2.5}\right)^2}}{\int_0^{50}e^{-\left(\frac{\tau-25}{2.5}\right)^2}d\tau}duds\,.
\end{equation}
Note that the term $U_\ell\ast K_{B'}(t)$ models the time-varying WSF. Consider a non-stationary noise $\Phi$, where for $n=1,\ldots,5,000$, $\Phi(n)$ follows an ARMA$(1, 1)$ process with Student's t4 innovation process, and for $n=5,000,\ldots,10,000$, $\Phi(n)$ i.i.d. follows a Student's t5 distribution. We assume $\Phi$ is independent of $X_1$ and $X_2$
The simulated signal is defined over $[0,50]$ with the sampling rate $200$Hz so that 
\begin{equation}
Y_1(n):=s_1(n/200)+\sigma \Phi(n), 
\end{equation}
where $n=1,\ldots,10,000$, $s_1(t):=A_1(t)x_1(t)$, $\sigma\geq0$ is determined by the desired signal-to-noise ratio (SNR), defined as $20\log_{10}\frac{\texttt{STD}(s_1)}{\texttt{STD}(\Phi)}$ with \texttt{STD} standing for the standard deviation,
$x_1(t)=D_1\cos(2\pi\phi_1(t)) + (u_1+u_2)\cos(2\pi\phi_2(t))+u_1\cos(2\pi\phi_3(t))$, $D_1\in(0,1]$ models the intensity of the fundamental component, and
$u_1,u_2$ are independent random variables uniformly distributed on $[0,1)$ that are independent of $X(t)$ and $\Phi(n)$. {See Figure \ref{figure an example of simu sig of 0 dB} in Section \ref{section more numerical simulation} for a realization of $Y_1(t)$.}

Then, we realize independently $Y_1(t)$ over $[0,50]$ for $100$ times.
For a fair comparison with existing algorithms, we utilize the provided code from the authors \cite{Colominas_Meignen_Pham_2020,Laurent_Meignen_2021}, following their recommended guidelines.\footnote{The code for \textsf{FM-RD} is obtained through private communication with the authors. The code for \textsf{RRP-RD} is available at \url{https://github.com/Nils-Laurent/RRP-RD}.} However, since these algorithms are not tailored to handle WSFs, we need to make minor adjustments and assume the number of harmonics is known.  In the case of \textsf{FM-RD}, we incorporate the \textit{peeling} step (\ref{peeling scheme formula}) as a post-processing procedure. Essentially, we run \textsf{FM-RD} iteratively and apply TFR masking three times. Among the three extracted curves, we identify the one with the lowest frequency value as the result for the fundamental component's ridge.
Regarding \textsf{RRP-RD}, we concurrently detect three relevant ridge portions on the TF domain and extract the corresponding ridges. From these, we choose the ridge associated with the lowest frequency value as the outcome. If \textsf{RRP-RD} is unsuccessful in identifying three ridge portions, we rerun it to detect two RRPs. If this attempt also fails, we proceed with a single RRP. Finally, we assign the ridge with the lowest frequency as the final outcome.

Let $c_1^{(\square)}$ be the detected ridge by the algorithm $\square$, which could be \textsf{S}, \textsf{FM}, \textsf{RRP} or \textsf{MH}, standing for the \textsf{Single-RD}, \textsf{FM-RD}, \textsf{RRP-RD} or \textsf{SAMD-MHRD}.
To compare different algorithms, we consider the following metric:
\begin{equation}
\delta^{(\square)}:=\|\phi_1'-c_1^{(\square)}\|_2/\|\phi_1'\|_2\,.
\label{err definition in simulation part}
\end{equation}
We run the Wilcoxon signed-rank test on $100$ realizations to test the null hypothesis that the median of $\delta^{(\square)}-\delta^{(\textsf{MH})}$ is zero, with a significance level of $0.05$ and Bonferroni correction. 

The results are depicted in Fig. \ref{simulation A}, illustrating the distributions of ${\delta^{(\square)}}$ for a fixed $D_1$ (rows: $D_1=0.1$, $D_1=0.2$, $D_1=0.5$) and a constant noise level (columns: noise-free, SNR = 5 dB, SNR = 0 dB) with asterisks indicating instances where $p< 0.05/27$. The outcomes reveal that in the absence of noise, all algorithms yield comparable results. However, when encountering scenarios with a weak fundamental component and noise interference, \textsf{FM-RD} and \textsf{RRP-RD} exhibit limitations compared to \textsf{SAMD-MHRD}. {More numerical results for \textsf{RRP-RD} with different number of ridge portions can be found in Section \ref{section more numerical simulation}.}

\begin{figure}[hbt!]
\centering
\includegraphics[width=0.95\textwidth]{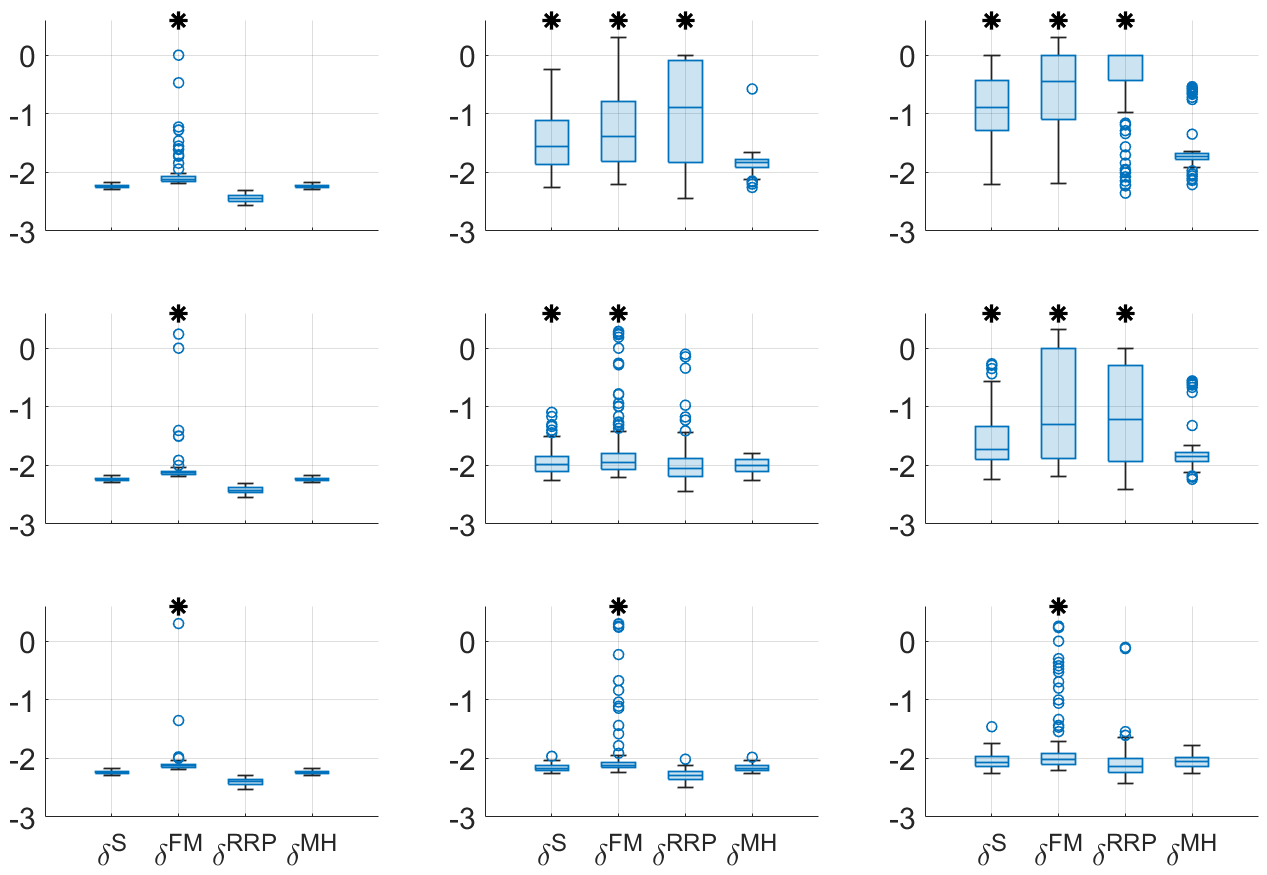}
\caption{
Log-scale distributions of $\delta^{(\square)}$ for various RD algorithms.
Top to bottom rows correspond to $D_1=0.1$, $D_1=0.2$, and $D_1=0.5$, respectively.
From left to right columns: noise-free, SNR = 5 dB, and SNR = 0 dB.}
\label{simulation A}
\end{figure}

\subsection{Multiple IMT functions ($L>1$)}\label{subsection numerics L>1}
Following the same notations used in Section \ref{section simulation L=1}, we consider the case when $L=2$.
Denote $X_3$ and $X_4$ as two independent copies of the smoothed standard Brownian motion and are independent of $X_1$, $X_2$, $\Phi$, $u_1$ and $u_2$.
Over $[0,50]$, set two random processes
$$
A_2(t):=e^{-\left(\frac{t-40}{25}\right)^{1.8}}
\left(3\int_0^t\frac{|X_3(s)|}{\|X_3\|_{L^\infty[0,50]}}ds+2.3\right)
$$
and
$$
\phi_2(t)=\phi_0^*(t)
+(\phi(t)-\phi_0(t))
+\int_0^t\frac{X_4(s)}{\|X_4\|_{L^\infty[0,50]}}ds\,,
$$
where $\phi_0^*(t)=2.33t+0.2t^2$.
The second simulated signal is defined over $[0,50]$ with the sampling rate $200$Hz so that
\begin{equation}
Y_2(n):=s_1(n/200)+s_2(n/200)+\sigma \Phi(n),
\end{equation}
where $s_2(t):=A_{2}(t)x_2(t)$, $x_2(t)=D_2\cos(2\pi\phi_1^*(t)) +(u_1^*+u_2^*)\cos(2\pi\phi_2^*(t)) +(u_1^*)\cos(2\pi\phi_3^*(t))\big]$, $D_2\in(0,1]$ models the intensity of the fundamental component, $u_1^*$ and $u_2^*$ are the independent copies of $u_1$ and $\sigma\geq0$ is again determined by the desired SNR. Then we realize 100 copies of $Y_2$ for each selected SNR.

See Fig. \ref{fig SAMD example} for the RD results and a decomposition example of \textsf{SAMD-MHRD}. Both IMT functions are successfully recovered. We then run \textsf{SAMD-MHRD} and evaluate the reconstruction error of the $j$-th IMT function as
\begin{equation}
\delta^{[q]}_j=\|s_j-\widehat{s}^{[q]}_j\|_2/\|s_j\|_2\,,
\label{RMSE of the recon using SAMD-MHRD}
\end{equation}
{where the superscript $q$ indicates the $q$-th iteration. See Fig. \ref{fig_simu: two components0} for the comparison of $\delta^{[1]}_j$ and $\delta^{[3]}_j$ under different SNRs.} As SNR decreases, errors increase. Additionally, the second IMT exhibits significantly larger recovery error than the first IMT, as tested with the Wilcoxon test at a significance level of $0.005$, which is reasonable due to its dependence on the harmonics of the first IMT. Iteration improvement is statistically significant except for the second IMT at high noise levels.

\begin{figure}[hbt!]
\centering
\includegraphics[width=1\textwidth]{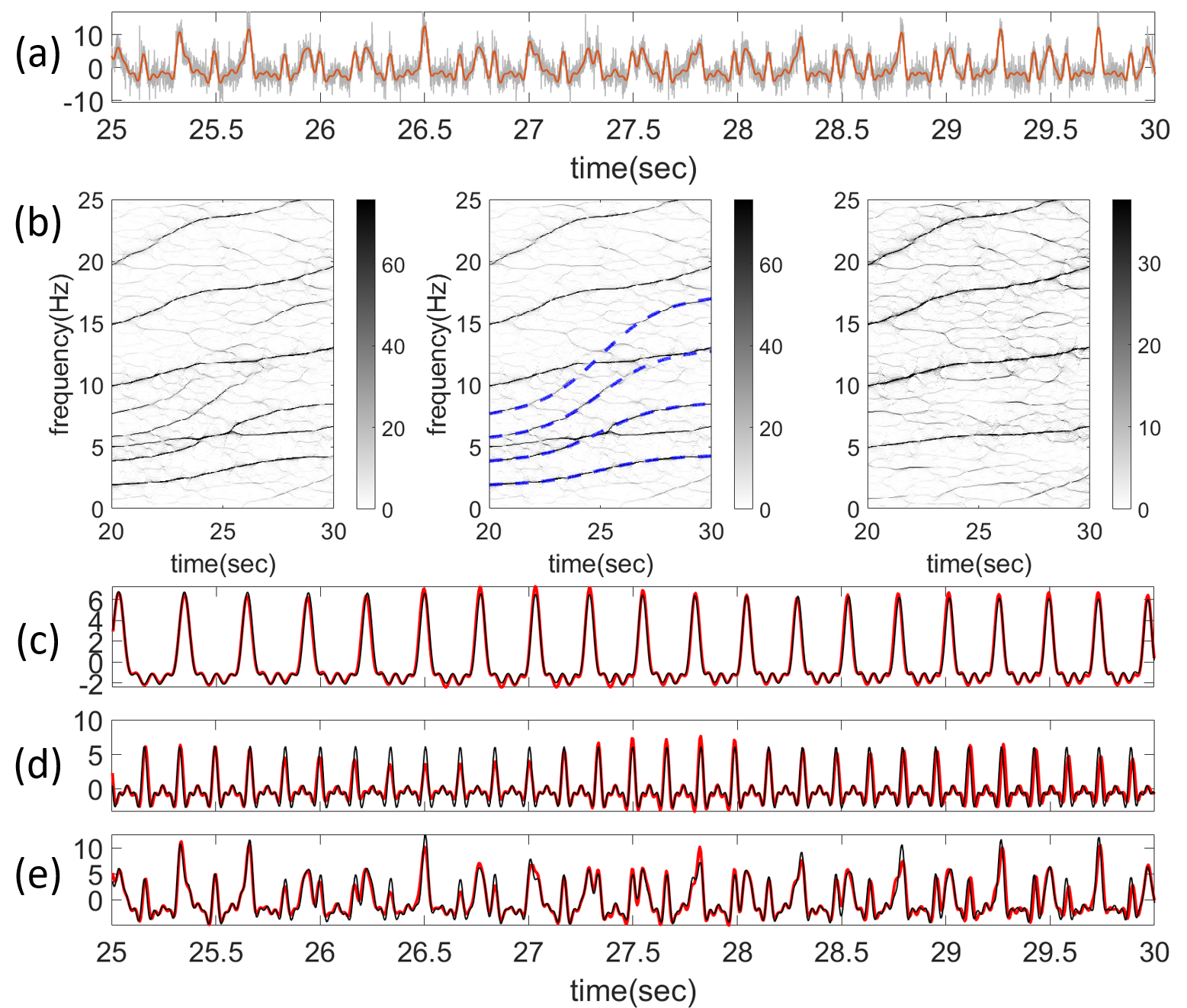}
\caption{(a) A signal fulfilling the ANHM with $L=2$ contaminated by noise $\Phi$ with an SNR of 5dB is shown as the grey curve. The clean signal is superimposed as the red curve.
(b)-Left: the 2nd-order SST of the noisy signal.
(b)-Middle: the true IFs of the first three harmonics of the first IMT function are superimposed as blue-dashed curves.
(b)-Right: the 2nd-order SST of the extracted first IMT function by \textsf{SAMD-MHRD} with 3 iterations.
(c) and (d): The reconstructed first and second IMT functions are shown as the red curves and the true IMT functions are superimposed as the black curves.
(e): The superposition of the reconstructed IMT functions is shown as the red curve and the clean signal is superimposed as the black curve.
\label{fig SAMD example}}
\end{figure}

\begin{figure}[hbt!]
\centering
\includegraphics[trim=100 0 100 0, clip,width=1\textwidth]{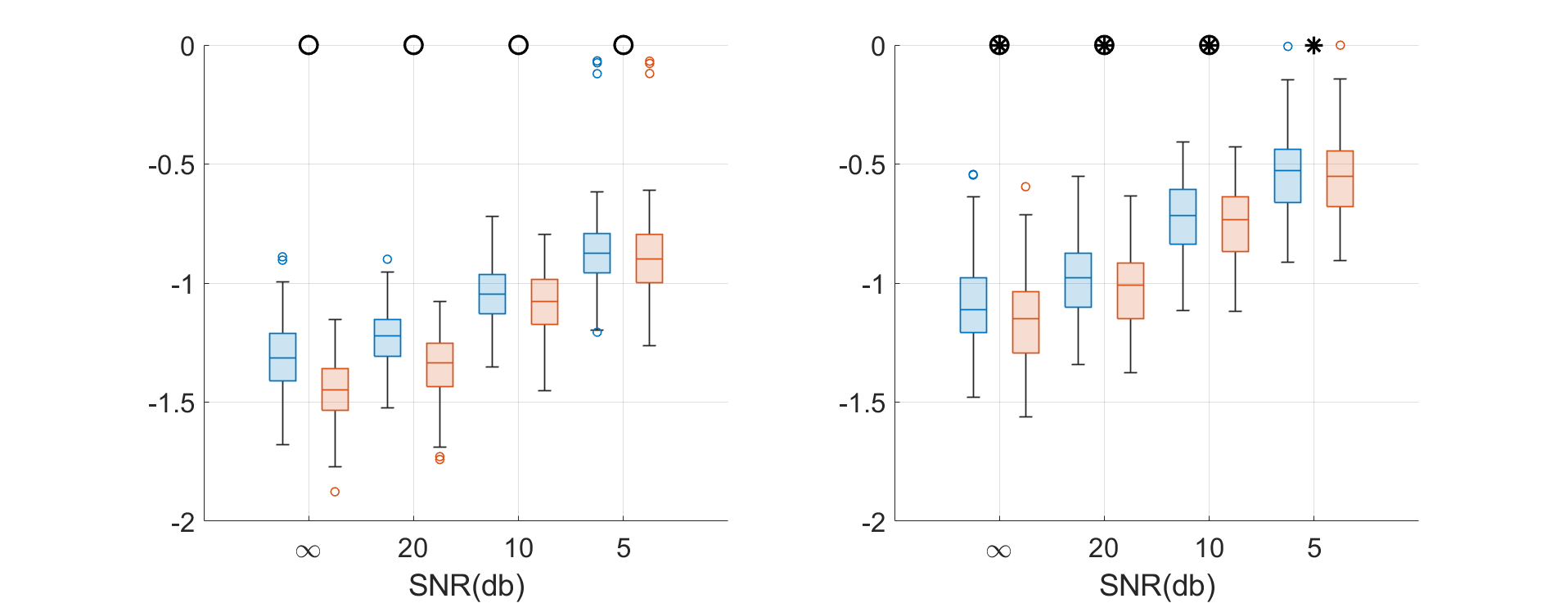}
\caption{{Errorbars of 100 realizations of $\log(\delta_1^{[q]})$ (left) and $\log(\delta_2^{[q]})$ (right), where $\delta_j^{[q]}$ is defined in \eqref{RMSE of the recon using SAMD-MHRD}.
Blue: first iteration ($q=1$), pink: third iteration ($q=3$).}
Asterisk: recovery error of IMT2 is significantly greater than that of IMT1 at the same noise level.
Circle: $\log(\delta_1^{[1]})$ is significantly greater than $\log(\delta_1^{[3]})$ at the same noise level.}
\label{fig_simu: two components0}
\end{figure}

\section{Application to walking prediction from IMU}\label{section real signals app}

We now apply the proposed algorithm to detect the walking activity from the IMU signal. The data obtained from the triaxial accelerometer signals are denoted as $f^{[\texttt{loc}]}_x(t)$, $f^{[\texttt{loc}]}_y(t)$, and $f^{[\texttt{loc}]}_z(t)$, where \texttt{loc} indicates the sensor's placement location, and $x$, $y$, and $z$ represent the axes. We focus on the {\em physical activity} signal, defined as
\[
Y^{[\texttt{loc}]}(t)=\sqrt{[f^{[\texttt{loc}]}_x(t)]^2+[f^{[\texttt{loc}]}_y(t)]^2+[f^{[\texttt{loc}]}_z(t)]^2}\,.
\]
To detect the walking activity, we modify the ANHM model in \eqref{Model:equation2} with $L=1$ slightly to account for $Y^{[\texttt{loc}]}(t)$. Specifically, we express it as follows:
\begin{equation}
\label{Model:equation3}
Y^{[\texttt{loc}]}(t) = \sum_{q=1}^Q \sum_{j=1}^D a_{q,j}(t)\cos(2\pi\phi_{q,j}(t))\chi_{I_q}+\Phi(t)\,,
\end{equation} 
where $I_q$, $q=1,\ldots,Q$, are disjoint connected intervals in $\mathbb{R}$ that describes the wax-and-wane effect of walking; that is, a subject might walk for a bit, stop for a while, and continue to walk. Physically, $\sqrt{\sum_{j=1}^\infty a_{q,j}(t)^2}$ reflects the intensity of the walking activity over $I_q$, with both strength and pattern exhibiting variations between different steps. The walking activity encompasses a fusion of activities across body locations, resulting in the dependence of $a_{q,j}(t)$ on sensor placement and synchronization of $\phi_{q,1}(t)$ across sensors. Even within the same location, $\phi_{q,1}(t)$, $\phi_{q,1}'(t)$, and $a_{q,j}(t)$ can differ between different walking bouts. For instance, the walking pattern on a flat surface during interval $I_i$ might differ from the pattern during stair climbing in interval $I_j$. This model reduces to \eqref{Model:equation2} with $L=1$ when $Q=1$, and has been used in \cite{wu2023application} for cadence estimation.

Finally, we make an assumption regarding the durations of walking intervals. The definition of ``walking'' lacks universal consensus \cite{urbanek_prediction_2018}, leading to questions about whether a mere one or two steps qualify as ``walking'' and what is the number of continuous steps needed to classify an interval as ``walking''. While our work does not aim to address this scientific matter, we provide a mathematical definition based on the properties of the TF analysis tools we utilize. Building upon (C1)-(C8), we introduce the following assumption:
\begin{itemize}
\item[(C9)]
For each $q=1,\ldots,Q$, we have $|I_q|>C/\phi_{q,1}'(t)$ for all $t\in I_q$ and a sufficiently large $C>1$. 
\end{itemize}
That is, we require $|I_q| > \phi_{q,1}^{-1}(2\pi C) - \phi_{q,1}^{-1}(0)$, as the phase advances by $2\pi$ after each cycle. Empirical evidence suggests that, to reliably estimate the IF of an oscillatory signal using STFT or SST, the window function should span around $8$ to $10$ oscillatory cycles. Therefore, we designate a process as ``walking'' only if a subject performs more than $8$ continuous steps by setting $C=8$. Note that $I_q$ can vary depending on the walking pace; faster cadence leads to shorter $I_q$. This definition aligns with the concept of {\em sustained harmonic walking} (SHW) presented in \cite{urbanek_prediction_2018}, where SHW involves walking for at least $10$ seconds (in line with C9) with minimal variability in step frequency (in line with C2).

In practice, $I_1,\ldots,I_Q$ are unknown. The main signal processing mission we consider in this section is estimating $I_1,\ldots,I_Q$ from one realization of the random process $Y$ (or recorded accelerometer signal).

\subsection{Database}
\label{Section: measurements}
We analyze the Indiana University Walking and Driving Study (IUWDS)\footnote{\url{https://physionet.org/content/accelerometry-walk-climb-drive/1.0.0/}} dataset, collected in 2015 from $n=32$ individuals (19 females) aged 21 to 51 years. This study took place in a semi-controlled environment, where participants were directed to follow a route involving walking on flat ground and climbing stairs. Data was recorded using ActiGraph GT3X+ accelerometers at 100 Hz. Fig. \ref{fig2} provides an example of this dataset. The study was ethically approved by the IRB of Indiana University, and participants provided informed written consent. The dataset comprises recordings from four locations: wrist (\texttt{wr}), hip (\texttt{hi}), left ankle (\texttt{la}), and right ankle (\texttt{ra}), alongside expert annotations.

\subsection{Our proposed index for walking detection}
To harness the distinctive properties of WSF, we propose employing \textsf{SAMD-MHRD} for devising a walking detection metric. Denote the extracted IFs of the harmonics by \textsf{SAMD-MHRD} as $\mathbf{c}\in \mathbb{R}^{K\times N}$, with $K\in \mathbb{N}$ representing the user-specified number of harmonics.
We introduce the {\em SST walking strength index} (SST-WSI) at time $\ell\Delta_t$ (where $\ell=1,\ldots,N$) as follows:
\begin{align*}
{\textsf S}_{\texttt{SST}}(\ell\Delta_t):=&\,\sum_{k=1}^Q\bigg| \sum_{j\in B^{(k)}_\ell}\mathbf{S}(\ell\Delta_t,\,j)\bigg|
\left[ \sum_{q=1}^{M} \left|\mathbf{S}(\ell\Delta_t,\,q)\right| \right]^{-1},
\end{align*}
where $\mathbf{S}$ is the SST of $Y^{[\texttt{loc}]}(t)$, 
\begin{align}
B^{(k)}_\ell=&\,\left\{q: \,\max\{1,\,e_k^{\top}\mathbf{c}(\ell)-b\}\leq q  \right.\\
&\quad\left.\leq \min\{M,\,e_k^{\top}\mathbf{c}(\ell)+b\}\right\}\subset [M],\nonumber
\end{align}
$\mathbf{c}$ is determined by \textsf{SAMD-MHRD}, 
$b\in\mathbb{N}$ is the bandwidth chosen by the user, and $\left[ \frac{f_s}{2M} \sum_{q=1}^{M} \left|\mathbf{S}(\ell\Delta_t,\,q)\right| \right]^{-1}$ is a normalization factor.
We set $b$ so that $B^{(k)}_\ell$ and $B^{(k-1)}_\ell$ are sufficiently separate. 
An intuitive interpretation of SST-WSI is capturing the ratio of energy associated with the walking activity.  Specifically, the extracted curves ${e_1^{\top}\mathbf{c},\cdots,e_K^{\top}\mathbf{c}}$ correspond to the most prominent features in the TFR, effectively capturing instances of walking activity. See Section \ref{section justification of SST-WSI} for a justification of SST-WSI. Since the WSF depends on the subject, sensor location and the walking pattern, {$D$} needed for the walking activity detection depends on these parameters. To simplify the discussion, we fix an empirical {$D$} in this work.

\subsection{Results}
The number of subjects {is $32$, with the numbers of walking and non-walking segments in the database} $21,245$ and $38,623$.
We consider four existing walking detection indices, including the \textit{SHW index}, the \textit{Entropy Ratio index}, the Hilbert transform based index called the Hilbert-WSI, and the index of freezing of gait called FOG-WSI \cite{moore2008ambulatory}. Details of these indices' implementation can be found in Section \ref{section existing WSIs} in the supplementary materials. For the \textit{SST-WSI} index, we set $D=8$ and $b=0.08$Hz.

We use Leave One Subject Out Cross Validation (LOSOCV) to evaluate the performance. In each round, one subject is used as the test case while the others determine the threshold. This threshold is then applied to the test subject, and performance is recorded. Results are shown in Fig. \ref{WSI-wrist-boxchart}. For Hilbert-WSI, some validation results give zero true positives, and we define its F1-score to be 0. We use the Wilcoxon signed-rank test ($p=0.05$) with Bonferroni correction to assess significance between indices. SST-WSI and Entropy-Ratio are the top-performing indices, with SST-WSI consistently better.

Finally, we combine SST-WSI and Entropy-Ratio using a support vector machine. Fig. \ref{WSI-wrist-boxchart} presents the {LOSOCV} results, indicating that this combination outperforms other indices. Additional results for different sensor placements can be found in Section \ref{more IMU location comparison results}.

\begin{figure}[hbt!]
\centering
\includegraphics[width=1\textwidth]{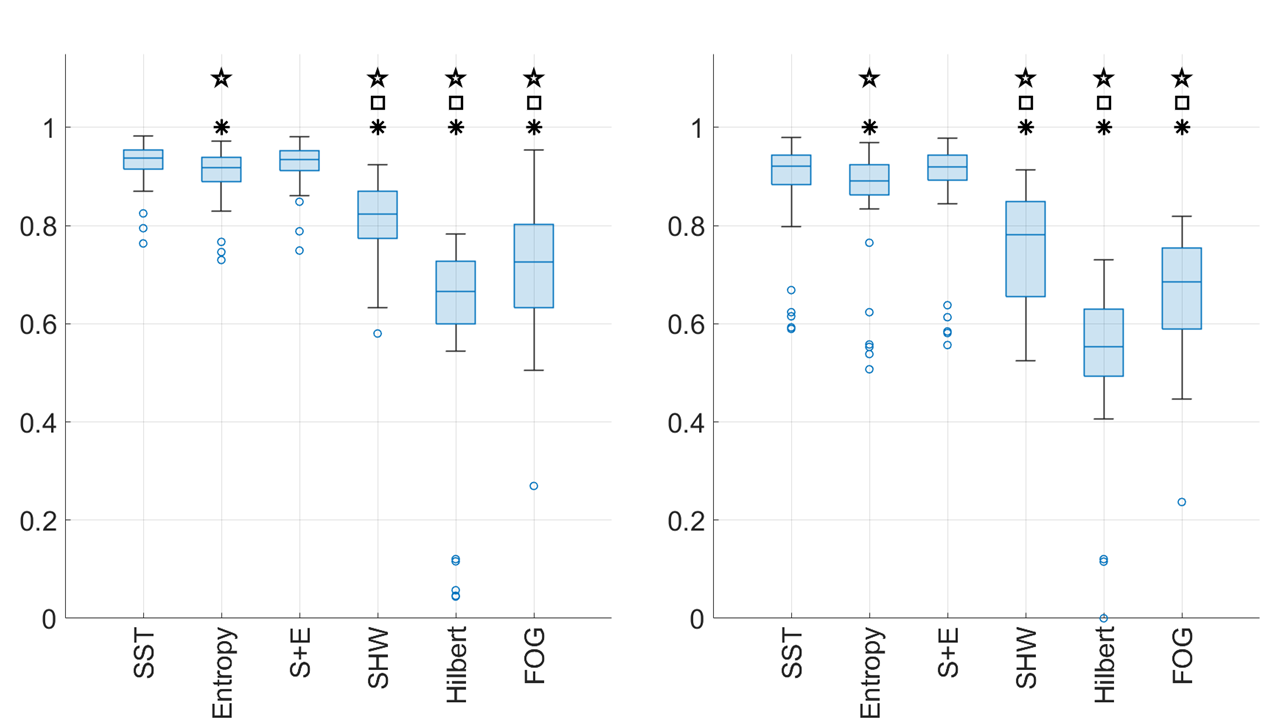}
\caption{The performance records of conducting LOSOCV on analyzing the WRIST signal.
Left: The accuracy distributions.
Right: The F1-score distributions. 
The asterisk (resp. square and pentagram) marks indicate the performance is lower than that of SST-WSI with statistical significance (resp. Entropy-Ratio group and SST-WSI+Entropy-Ratio, denoted as S+E) with the $p$-value less than $0.05/18=0.0028$. \label{WSI-wrist-boxchart}}
\end{figure}

\section{Discussion and Conclusion}\label{section discussion conclusion}

We introduce a novel RD algorithm, \textsf{SAMD-MHRD}, designed for both STFT and SST, to effectively manage nonstationary signals featuring intricate non-sinusoidal WSFs. Numerical investigations and real-world applications showcase its practical promise. This endeavor can be seen as an extension of \cite{Wu:2013}, filling the RD gap toward signal processing missions.

Naturally, one may question its adaptability to other TFRs, such as the commonly used CWT. The CWT's inherent affine structure can be limiting for oscillatory signals with non-sinusoidal WSFs, as it requires a small scale to capture higher-order harmonics while dealing with broad spectral support, leading to spectral interference among these harmonics. On the other hand, \textsf{SAMD-MHRD} may be applicable to other TF analysis algorithms that can separate harmonics within their TFRs. Further exploration in this area is needed.

Another strategy for handling non-sinusoidal WSFs involves using de-shape STFT \cite{lin2016waveshape}. This method aims to mitigate the impact of harmonics in complex situations and could potentially be used to extract the fundamental IF of each IMT function. However, its success relies on the presence of a strong fundamental component. In cases where the fundamental component is weak or absent, preprocessing steps like rectification or nonlinear transforms \cite{steinerberger2022fundamental}, or the application of a comb filter as explored in \cite{urbanek_prediction_2018}, might be necessary to enhance the fundamental IF before applying de-shape STFT. Note that the fundamental component can be absent since $s(t)=\sum_{k=1}^\infty \alpha_k \cos(2\pi kt+\beta_k)$, where $\alpha_k\geq 0$ and $\beta_k\in [0,2\pi)$ is $1$-periodic if the greatest common divisor of $\{k\mid \alpha_k>0\}$ is $1$. We plan to explore the development of an RD algorithm based on this concept in our future work.

Several challenges and open questions remain in this study. Firstly, we assume knowledge of $L$ to simplify the analysis. Extending existing tools \cite{SucicSaulig2011,SauligPustelnik2013,Laurent_Meignen_2021,ruiz2022wave} to handle more general WSF cases is a promising avenue. Secondly, our assumption that the IFs of fundamental components do not overlap needs refinement to address the common scenario of mode mixing. Extending the RD algorithm to accommodate the synchrosqueezed chirplet transform \cite{chen2023disentangling} could be a solution for this challenge.
The walking detection problem is essentially a change point detection problem. We can potentially generalize the existing bootstrapping-based change point detection algorithm by integrating information related to IF and IA and techniques from \cite{ruiz2022wave} to yield a more accurate index. These promising directions will be explored in our future work.

\section*{Acknowledgement}

{The authors would like to thank the anonymous reviewers for their constructive feedbacks.}

\bibliographystyle{plain}
\bibliography{Wristcadence}

\appendix

\section{Numerical implementation of TF analysis tools}\label{section numerical implementation of TFA}

We summarize the numerical implementation of STFT and SST in this section. Suppose the signal $f$ is discretized with a sampling rate $f_s$ over the time interval $[0,T]$, where $T>0$, denoted as  $\mathbf{f}$; that is,
$\mathbf{f}=(f(0),f(1/f_s),\cdots,f(T/f_s))$ with $N=\lfloor T/f_s\rfloor$, where $\lfloor x\rfloor$ means the largest integer smaller or equal to $x\in \mathbb{R}$. Denote $\Delta_t=1/f_s>0$ as the discretization grid in the time domain. Denote $\Delta_\xi>0$ as the frequency resolution chosen by user. 
Then, the discretization of STFT of $f$ is given by
\begin{equation}
    \mathbf{V}(\ell,q) = V_f^{(h)}(\ell\Delta_t,q\Delta_\xi)\,,
    \label{STFT}
\end{equation}
where $\ell=1,\ldots,N$ is the total number of samples in the time domain and $q=1,\ldots,M$, where $M=\lfloor\frac{f_s}{2\Delta_\xi}\rfloor$ is the total number of ticks in the frequency domain. In other words, the discretization of the TFR determined by STFT is a complex matrix $\mathbf{V}\in \mathbb{C}^{N\times M}$. Note that we skip evaluating STFT at frequency $0$.

Numerically, this discretization is carried out in the following way. Discretize the window $h$ and denote it as $\mathbf{h}\in \mathbb{R}^{2K+1}$, where $K\in \mathbb{N}$ and $2K+1$ indicates the window length. A concrete example is discretizing the Gaussian with the standard deviation $\sigma>0$ over the interval $[-0.5, 0.5]$ with a discretization grid size of $\frac{1}{2K}$ as the window by setting
\begin{gather}
\mathbf{h}(k) = e^{-\frac{\left(\frac{k-1}{2K} - 0.5 \right)^2}{2\sigma^2}}
\end{gather}
for $k = 1, \ldots, 2K+1$. Note that the discretization grid size $\frac{1}{2K}$ may be different from $\Delta_t$. As the rule of thumb, we choose $K$ so that the window covers about $8\sim 10$ cycles if the data is not too noisy, and choose a larger $K$ so that the window covers more, like $15\sim 20$ cycles if the data is noisy. Usually, if the Gaussian function is taken as the window, we suggest to choose $\sigma$ so that the Gaussian function is ``essentially supported'' on $[-0.5, 0.5]$; that is, $\int_{-0.5}^{0.5} e^{-\frac{t^2}{2\sigma^2}}dt$ is close to $1$. We mention that it is possible to a time-dependent $\sigma$ to enhance the overall performance. For example, take the Re\'nyi entropy as a metric to determine the optimal $\sigma$ \cite{Sheu_Hsu_Chou_Wu:2017}. We skip this detail and refer readers to \cite{Sheu_Hsu_Chou_Wu:2017} for details. 
To implement $\mathbf{V}$, extend $\mathbf{f}$ beyond $1,\ldots,N$ so that $\mathbf{f}(l) := 0$ when $l < 1$ or $l > N$. It is possible to extend $\mathbf{f}$ by forecasting to avoid the boundary effect \cite{meynard2021efficient}, but we skip this technical detail to simplify the discussion.
Then, compute
\begin{equation}\label{definition STFT discretization}
\mathbf{V}(n, m) = \sum_{k=1}^{2K+1} \mathbf{f}(n + k - K - 1) \mathbf{h}(k) e^{-i2\pi \frac{(k-1)m}{2M}},
\end{equation}
where $n=1,\ldots,N$ is the index in the time domain and $m=1,\ldots,M$ is the index in the frequency domain.

The numerical implementation of SST is denoted as $\mathbf{S}\in \mathbb{C}^{N\times M}$, which is a nonlinear transform of STFT. Here we focus on its numerical implementation and refer readers to \cite{DaLuWu2011,Chen_Cheng_Wu:2014} for its theory in the continuous setup. 
The main idea beyond SST is utilizing the phase information hidden in STFT. Consider the discretization of the derivative of the window function $h$, which is denoted as $\mathbf{h}'\in \mathbb{R}^{2K+1}$. In the Gaussian window case, it is defined as
\[
\mathbf{h}'(k) = -\left(\frac{k-1}{2K} - 0.5 \right)\frac{\mathbf{h}(k)}{\sigma^2}\,.
\] 
Then, compute another STFT of $\mathbf{f}$ with the window $\mathbf{h}'$, denoted as $\mathbf{V}' \in \mathbb{C}^{N \times M}$, by
\begin{equation}
\mathbf{V}'(n, m) = \sum_{k=1}^{2K+1} \mathbf{f}(n + k - K - 1) \mathbf{h}'(k) e^{\frac{-i2\pi (k-1)m}{2M}}\,,
\end{equation}
where $n=1,\ldots,N$ is the index in the time domain and $m=1,\ldots,M$ is the index in the frequency domain. 
Next, compute the {\em frequency reassignment operator}, denoted as ${\Omega}\in \mathbb{C}^{N \times M}$, which is defined as
\begin{gather}\label{definition Omega}
{\Omega}(n, m) = \left\{
\begin{array}{ll}
-\Im\left(\frac{N}{2\pi(2K + 1)}\frac{\mathbf{V}'(n,m)}{\mathbf{V}(n,m)}\right) &\mbox{ when }|\mathbf{V}(n,m)|> \upsilon\\
-\infty&\mbox{ when }|\mathbf{V}(n,m)|\leq \upsilon\,,
\end{array}
\right.
\end{gather}
where $\Im$ is the operator of evaluating the imaginary part and $\upsilon > 0$ is a threshold chosen by the user. In practice, $\upsilon$ is set to avoid possible blowup situation that $|\mathbf{V}(n,m)|\approx 0$ and $|\mathbf{V}'(n,m)|$ is large, and can be set to $10$ times of the machine epsilon. ${\Omega}(n, m)$ gives the {\em instantaneous frequency} information of some IMT inside the signal at time $n\Delta_t$ and frequency $m\Delta_\xi$. See \cite{DaLuWu2011,Chen_Cheng_Wu:2014,sourisseau2019inference} for more technical details.
Finally, the SST of $\mathbf{f}$ is computed by 
\begin{gather}
\mathbf{S}(n, m) = \sum_{l;\, m= l- {\Omega}(n, m)} \mathbf{V}(n, l) \,,
\end{gather}
where $n=1,\ldots,N$ is the index in the time domain and $m=1,\ldots,M$ is the index in the frequency domain. In short, to detect if the input signal oscillates at frequency $m\Delta_\xi$, we find all STFT coefficients sharing the same instantaneous frequency information by searching over $\Omega$, sum them together, and put the result in the $m$-th entry of the $n$-th row of $\mathbf{S}$. This step is called ``squeezing''. Since the squeezing happens at a fixed time, this motivated the ``synchro'' part of the nomination SST.

In this paper, we denote the discretized TFR of $f$, either determined by STFT or SST, as $\mathbf{R}\in \mathbb{C}^{N\times M}$.

\section{More details of numerical simulation}\label{section more numerical simulation}

{In this section, we provide more details of numerical simulation in Section \ref{section numerics}. A realization of the simulated signal with a signal-to-noise ratio of 0 dB, as discussed in Section IV.A, is shown in Figure \ref{figure an example of simu sig of 0 dB}.
For RRP-RD, a key parameter is the number of ridge portions, set to $3$ in the main article. However, noise can cause ridges to split into different ridge portions. We extended the simulation to test various numbers of ridge portions, with results shown in Figure \ref{log-RMSE of more RRP-RD}. We used Wilcoxon's signed-rank test with a significance level of $0.05$ and Bonferroni correction to compare RRP-RD performance across different ridge portions. The results show that performance with 2 ridge portions is significantly lower than with $3$ or $4$. Increasing from $3$ to $4$ ridge portions slightly worsens performance, particularly when $D_1$ is small. In some cases the difference is significant.}

\begin{figure}[hbt!]
\centering
\includegraphics[width=0.95\textwidth]{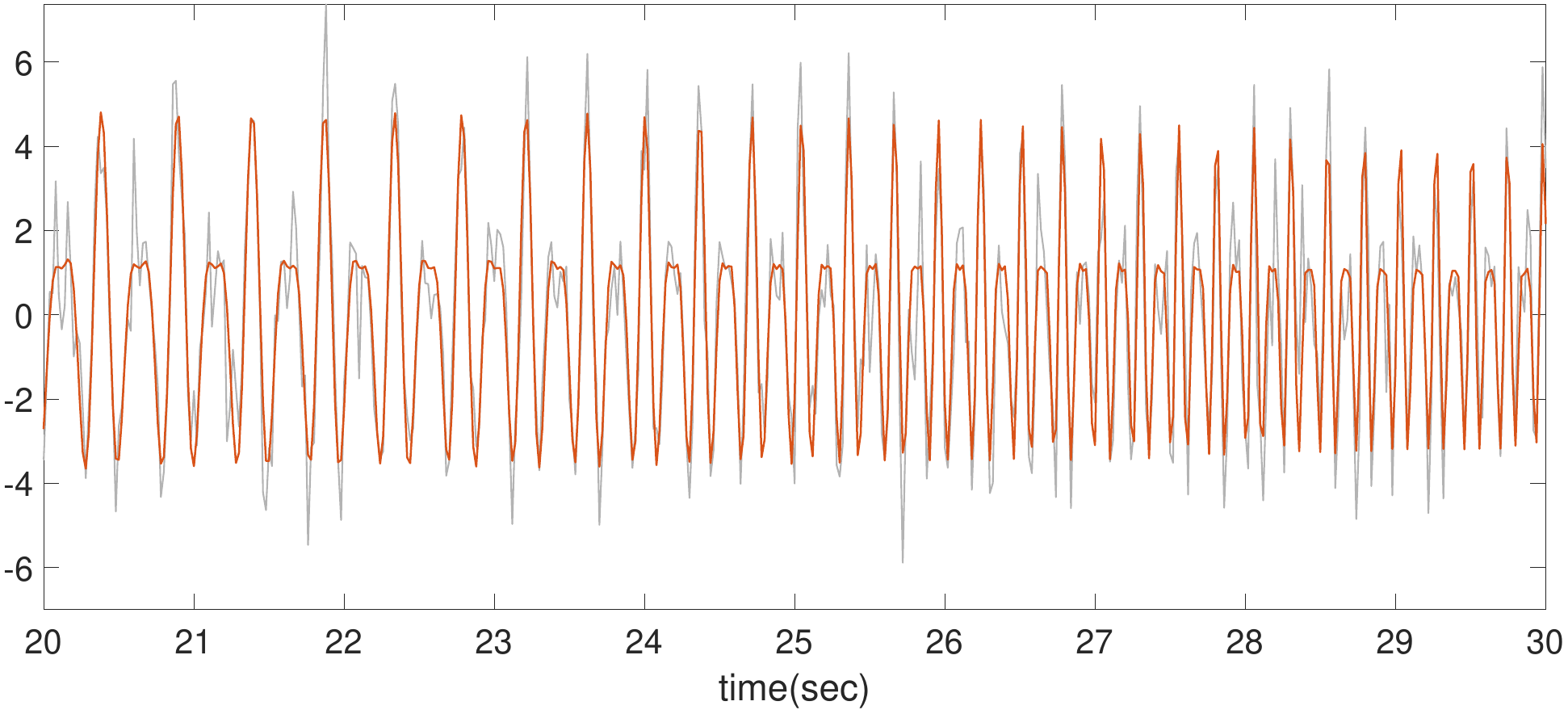}
\caption{{A realization of the simulated signal with a signal-to-noise ratio of 0 dB, as discussed in Section IV.A, is shown. The red curve is the clean signal, and the gray curve is the noisy signal. For better visualization, the segment from the 15th to the 35th second is displayed.}\label{figure an example of simu sig of 0 dB}}
\end{figure}

\begin{figure}[hbt!]
\centering
\includegraphics[width=0.95\textwidth]{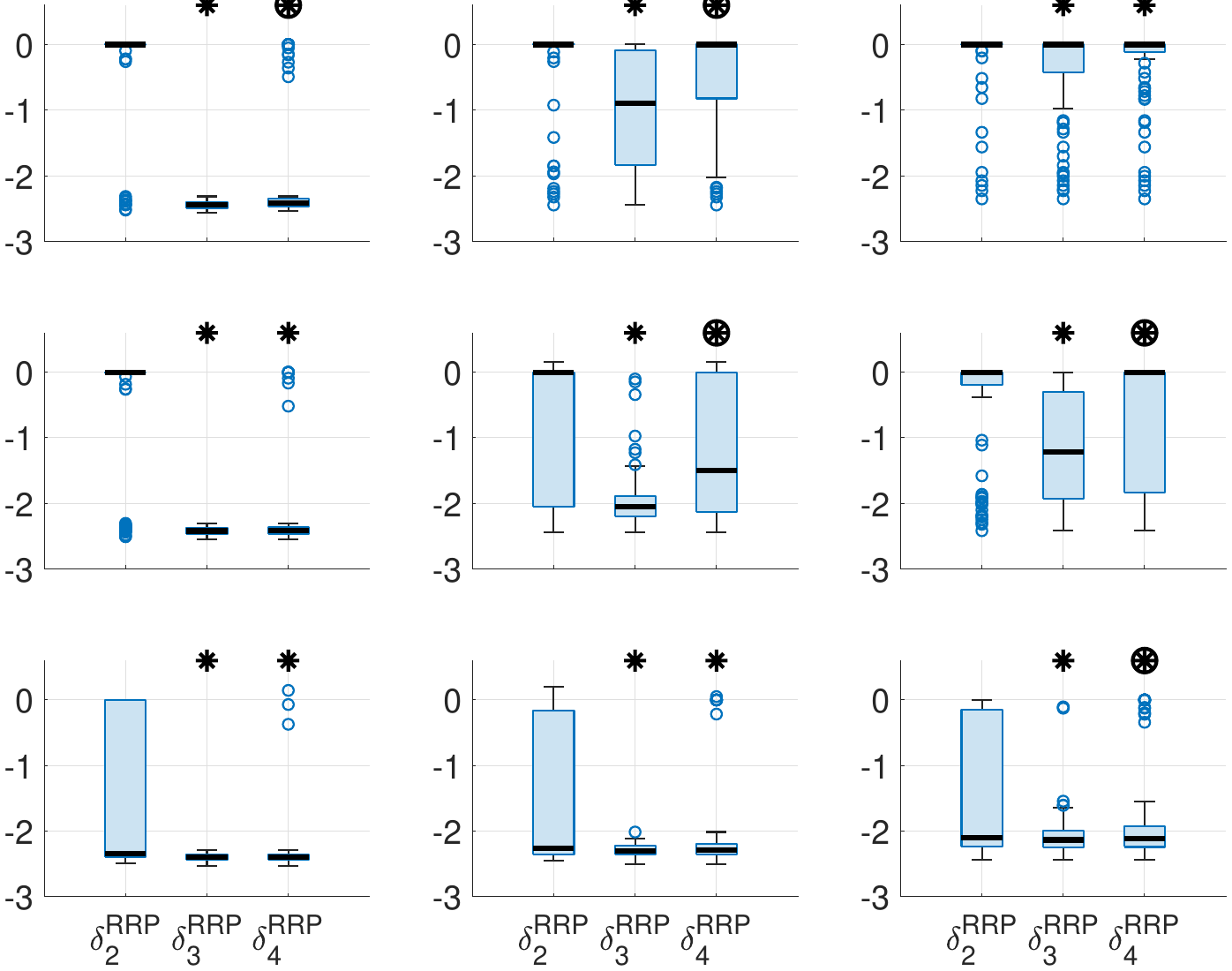}
\caption{{Additional log-RMSE distributions for the RRP ridge detection method with varying numbers of ridge portions. The box plot displays log-RMSE, with the median represented by a thick black bar.
The notation $\delta_p^{\textsf{RRP}}$ denotes the log-RMSE of the detected fundamental IF curve given by the RRP RD algorithm, with the number of ridge set to be $p$, where we have tested on $p=2$, $3$ and $4$.
Top to bottom rows: $D_1 = 0.1$, $D_1 = 0.2$ and $D_1 = 0.5$, respectively.
Left to right columns: noise-free, SNR = 5 dB and SNR = 0 dB, respectively.
The asterisk symbol indicates that the median of $\delta_p^{\textsf{RRP}}$, where $p=3,4$, is significantly lower than $\delta_2^{\textsf{RRP}}$. 
The circle symbol indicates that the median of $\delta_4^{\textsf{RRP}}$ is significantly higher than $\delta_3^{\textsf{RRP}}$.
}}
\label{log-RMSE of more RRP-RD}
\end{figure}

\section{Proof of Proposition \ref{prop: speed up algorithm}}\label{SIsection proof of prop: speed up algorithm}
\begin{proof} 
Let $\mathbf{h}\in [M]^{(K-1)\times N}$ so that its $k$-th row is $\mathbf{h}_k$. Then,
\[
\mathcal{L}_{c_1}^{(k+1)}(e_k^{\top}\mathbf{c}^*)\leq\mathcal{L}_{c_1}^{(k+1)}(e_k^{\top}\mathbf{h})=\max_{h:[N]\rightarrow[M]}\mathcal{L}_{c_1}^{(k+1)}(h)\,,
\]
where $e_k\in \mathbb{R}^{K-1}$ is a unit vector with $e_k(k)=1$.
Thus,
\begin{align}
\max_{\mathbf{c}:[N]\rightarrow[M]^{K-1}}\sum_{k=1}^{K-1}\mathcal{L}_{c_1}^{(k+1)}(e_k^{\top}\mathbf{c})
=&\,
\sum_{k=1}^{K-1}\mathcal{L}_{c_1}^{(k+1)}(e_k^{\top}\mathbf{c}^*)\nonumber\\
\leq&\,
\sum_{k=1}^{K-1}\mathcal{L}_{c_1}^{(k+1)}(e_k^{\top}\mathbf{h}).
\end{align}
Since $\mathbf{c}^*$ is a solution to (\ref{multiCurveExt:with Fundamental}), it follows that $\sum_{k=1}^{K-1}\mathcal{L}_{c_1}^{(k+1)}(e_k^{\top}\mathbf{c}^*)
=
\sum_{k=1}^{K-1}\mathcal{L}_{c_1}^{(k+1)}(e_k^{\top}\mathbf{h})$.
\end{proof}

\section{More details for walking prediction from IMU}

\subsection{Existing indices for walking detection}\label{section existing WSIs}

For a fair comparison, we consider several existing walking detection indices. 
In \cite{urbanek_prediction_2018}, the proposed method is to utilize the comb filters to detect the SHW state.
The idea is to use a set of ``comb teeth functions'' on the spectrogram to capture the most possible fundamental step-to-step frequency that has the maximum energy allocation on its multiples.
A comb teeth function is determined by its fundamental frequency $s\in D_f$, where $D_f$ is a set of candidate fundamental frequencies.
With the help of a comb filter,
they define a function $Y(t,s)=\max_{k=x,y,z}\{Y_k(t,s)\}$, $Y_k(t,s)$, where $s$ is the possible fundamental frequency $s\in D_f$ and $t\in \mathbb{R}$ is time, which describes the strength of the periodic content of the accelerometry signal.
Finally, the threshold $\delta$ that decides whether there is a SHW or not is determined by the density function of $\max_{s\in D_f}\{Y(t,s)\}$ for all SHW and non-SHW periods for all subjects.
At the end, the \textit{SHW index} is given by
\[
\hat{y}_\delta(t)=
\left\{
\begin{array}{l}
1,\;\text{if }\max_{s\in D_f}\{Y(t,s)\}>\delta \\
0,\;\text{otherwise,}
\end{array}
\right.
\]
where $\hat{y}_\delta(t)=1$ means the prediction state is SHW and $\hat{y}_\delta(t)=0$ means non-SHW. %We implement the \textit{SHW index} following all parameters suggested in \cite{urbanek_prediction_2018}. In brief,
For each given accelerometry signal, the STFT is computed using the Hanning window of length 10 seconds, the number of the harmonics is set to be 6,
and the range of the fundamental frequency $D_f$ was set to be 1.2 Hz to 4.0 Hz. The kernel density estimation is applied with the normal kernel function to give the estimated densities. We implement the \textit{SHW index} \cite{urbanek_prediction_2018} following all parameters suggested in \cite{urbanek_prediction_2018} by choosing the number of harmonics $n_m=6$, Hanning window with window length $\tau=10$ seconds for the STFT, and $71$ candidates for the step-to-step frequency $D_f=\{1.2\text{Hz},1.24\text{Hz},\cdots,4\text{Hz}\}$ to generate the statistics $\max_{s\in D_f}\{Y(t,s)\}$.

Another common idea to detect walking activity is using the entropy ratio. This idea is similar to that used in the parameter selection in Section \ref{section parameter selection}.
After extracting the IF curves $\mathbf{c}$ from the SST $\mathbf{S}$, we construct a masked TFR $\mathbf{S}_1$ following \eqref{peeling scheme formula}, and compute the sequence $q(\ell):=R_\alpha\left(\left|\mathbf{S}_1(\ell\Delta_t,\cdot)\right|\right)$. Similarly, we define an entropy sequence $p(\ell):=R_\alpha\left(\left|\mathbf{S}(\ell\Delta_t,\cdot)\right|\right)$ associated with $\mathbf{S}$. The pointwise ratio of these two sequences tells whether the harmonics exist at certain time segment or not. Hence, the \textit{Entropy Ratio index} is defined as
\begin{equation}
Q_{\alpha,\texttt{SST}}(\ell\Delta_t)
:=
\textsf{med}\left[\frac{p(\ell)}{q(\ell)},10\right],\;\ell = 1,\cdots,N.
    \label{entropy-ratio-index}
\end{equation}
The notation $\mathsf{med}[\cdot,10]$ is the median filter with a $10$ seconds bandwidth.
Intuitively, the index $Q_{\alpha,\texttt{SST}}(\ell\Delta_t)$ should be small if $\ell\Delta_t$ is during a walking period, since the frequency concentration measurement $p(\ell)$ of that time on the SST should be small, and the one for the curves-removed TFR, $q(\ell)$, should be relatively large. 
On the contrary, if the state is not in a walking period, then $Q_{\alpha,\texttt{SST}}$ is expected to be close to 1. To implement the \textit{Entropy Ratio} index, with the IF curve $\mathbf{c}$, we construct $\mathbf{S}_1$ with a bandwidth of $0.04$Hz and take $\alpha=2.4$.

The next one is based on a bandpass filter to extract the possible walking pattern via the Hilbert transform. The ratio of the energy between 0.5 and 3 Hz and the energy between 0.3 and 8 Hz is called the Hilbert-WSI. This index comes from capturing the energy associated with the fundamental frequency, and the spectral range is chosen to cover the common momentary speed of stride, which is suitable for our data. Another index used in evaluating the freezing of gait (FOG) in patients with Parkinson's disease is defined as the ratio of the energy between 0.5 and 3 Hz and the energy between 3 and 8 Hz \cite{moore2008ambulatory}. The index is called iFOG in \cite{moore2008ambulatory}, but for the purpose of unifying the notation, we call it FOG-WSI. Note that while FOG-WSI was not designed for the walking detection, its design can potentially be used for walking detection, and we considered it in this work.

\subsubsection{Justification of SST-WSI}\label{section justification of SST-WSI}
It is important to note that SST-WSI's design accounts for the expansion \eqref{Expansion f Fourier series}, and the influence of ${\textsf S}_{\texttt{SST}}(\ell\Delta_t)$ characterizes the walking strength at time $\ell\Delta_t$. Indeed, when $\ell\Delta_t\in I_j$ and $\frac{bf_s}{2M}=\epsilon^{1/3}$, where $\epsilon>0$ is from the slowly varying assumption for $a_{q,k}$ and $\phi'_{q,k}$, by \cite{DaLuWu2011,Chen_Cheng_Wu:2014} we have for $k= 1,\ldots,K$, 
\begin{align}
\frac{f_s}{2M} \sum_{q\in B^{(k)}_\ell}\mathbf{S}(\ell\Delta_t,\,q)
\approx a_{q,k}(\ell\Delta_t) e^{2\pi i\phi_{q,k}(\ell\Delta_t)}
\end{align}
is the reconstruction of the $k$-th harmonic when $\ell\Delta_t\in I_q$. Thus, the SST-WSI evaluates directly the walking strength since 
\begin{align*}
{\textsf S}_{\texttt{SST}}(\ell\Delta_t)&\,\approx\sum_{k=1}^K \left|a_{q,k}(\ell\Delta_t)e^{i2\pi \phi_{q,k}(\ell\Delta_t)}\right|^2\approx \sum_{k=1}^K a_{q,k}(\ell \Delta_t)^2\,,
\end{align*}
which is the strength of the walking activity at time $\ell\Delta_t$. 

\subsection{More results}\label{more IMU location comparison results}

The comparison of various walking detection indices when the accelerometer signal is recorded from the left ankle (right ankle and hip, respectively) is shown in Figure \ref{WSI-boxchart-la} (Figures \ref{WSI-boxchart-ra} and \ref{WSI-boxchart-hip}, respectively). Note that in the HIP signal experiment, all the validation results of the Hilbert-WSI group has zero True-Positive rate, which leads to $0$ F1 (marked by the red frame).

\begin{figure}[hbt!]
\centering
\includegraphics[width=0.95\textwidth]{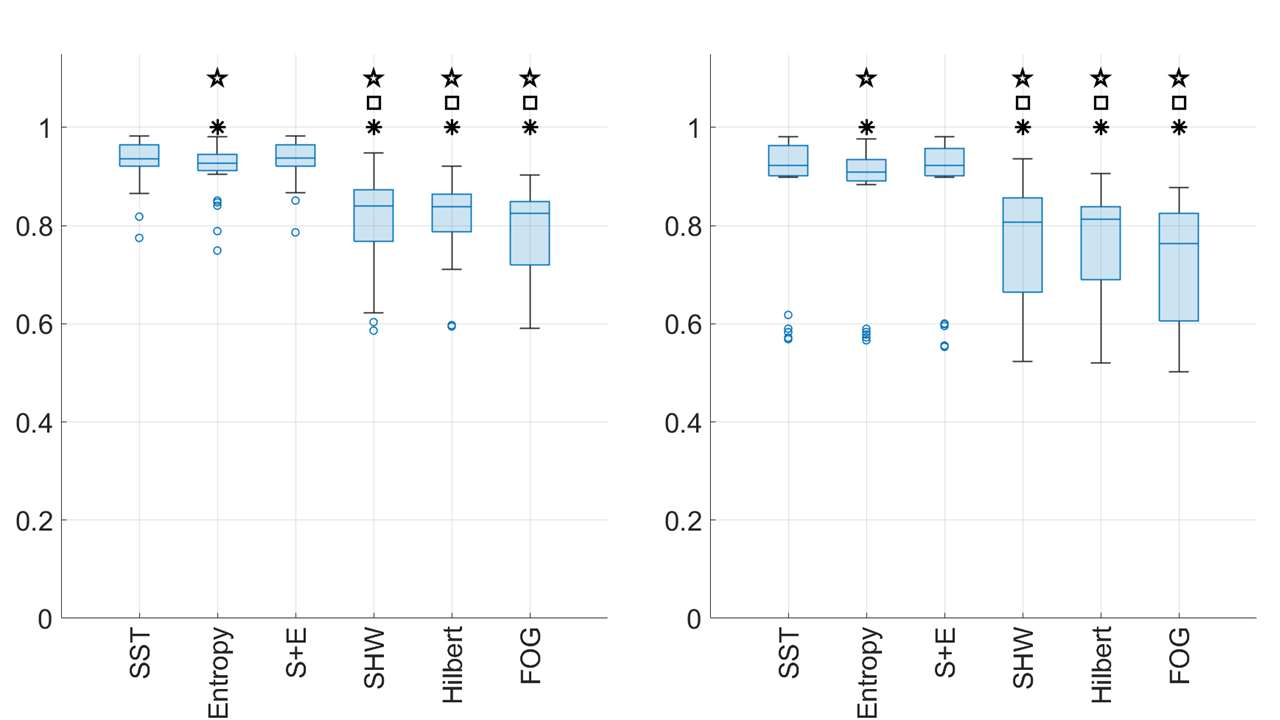}
\caption{The performance distributions of the LEFT-ANKLE signal.
\textbf{Left:} The accuracy distributions.
\textbf{Right:} The F1-score distributions.} 
\label{WSI-boxchart-la}
\end{figure}

\begin{figure}[hbt!]
\centering
\includegraphics[width=0.95\textwidth]{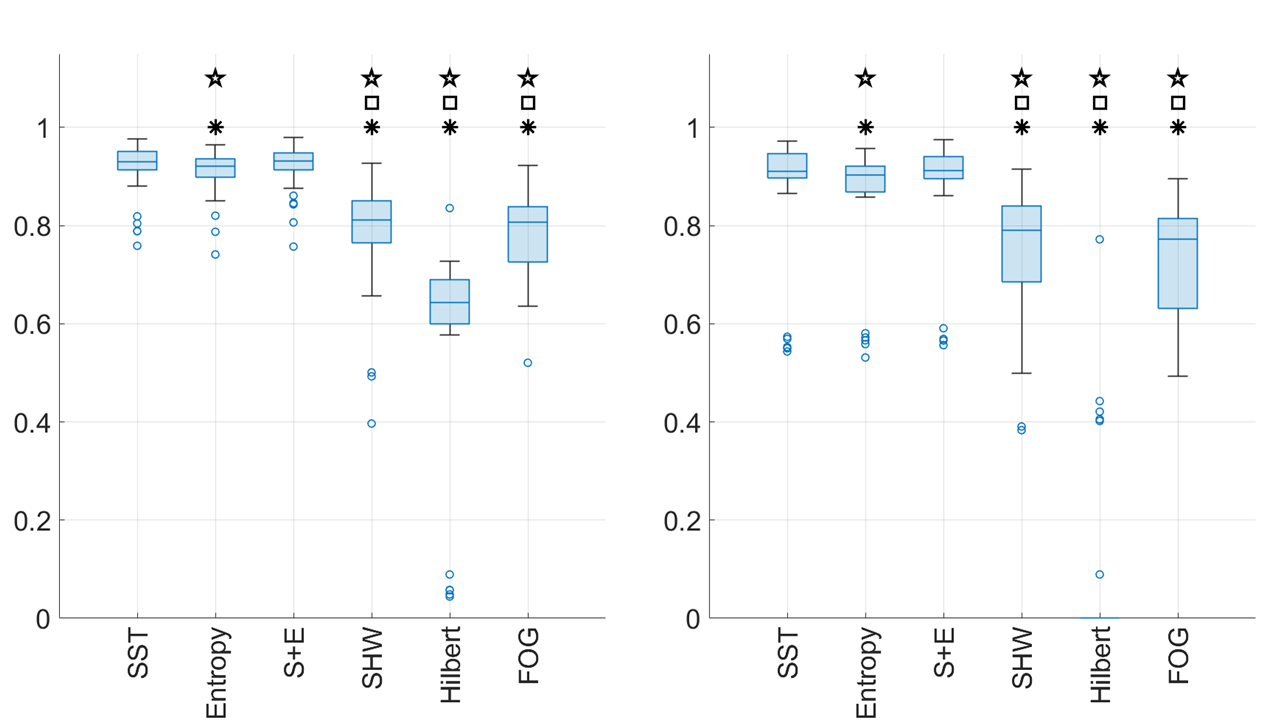}
\caption{The performance distributions of the RIGHT-ANKLE signal.
\textbf{Left:} The accuracy distributions.
\textbf{Right:} The F1-score distributions.}
\label{WSI-boxchart-ra}
\end{figure}

\begin{figure}[hbt!]
\centering
\includegraphics[width=0.95\textwidth]{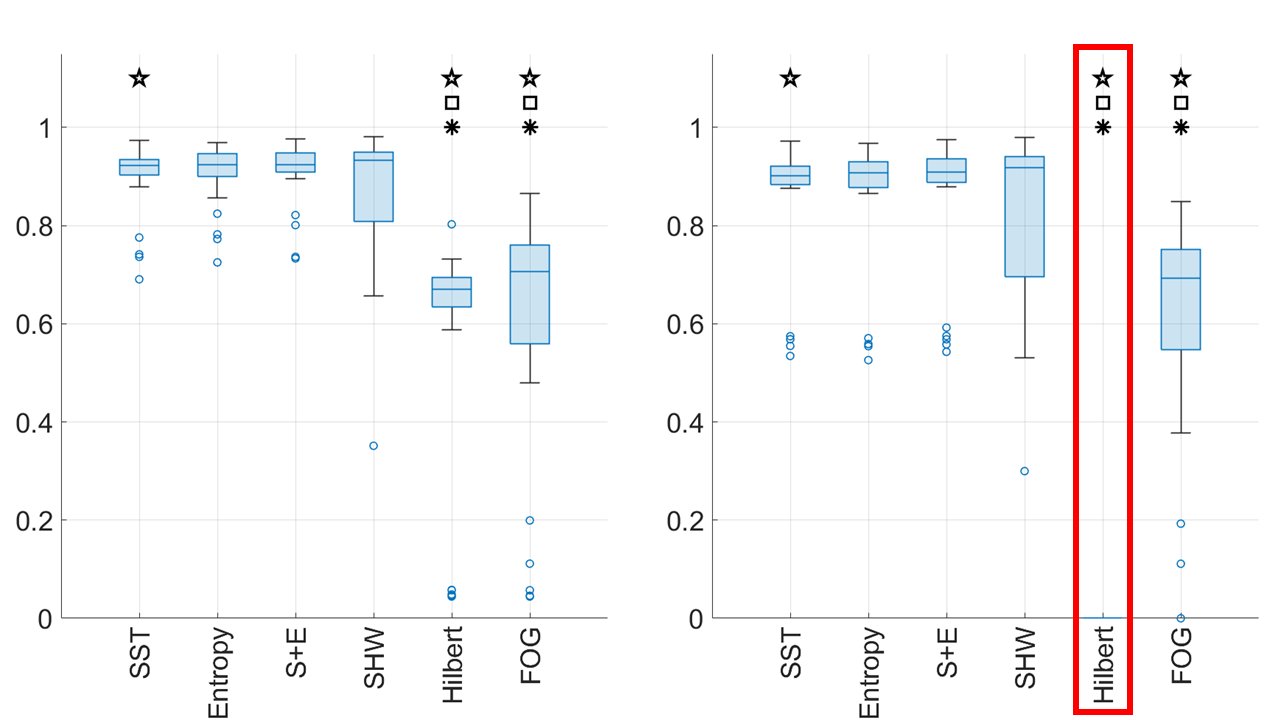}
\caption{The performance distributions of the HIP signal.
\textbf{Left:} The accuracy distributions.
\textbf{Right:} The F1-score distributions.}
\label{WSI-boxchart-hip}
\end{figure}

\end{document}